\documentclass{amsart}
\usepackage[utf8]{inputenc}
\usepackage{amsmath,amsthm,amssymb,bbm}
\usepackage{hyperref}
\usepackage{cleveref}

\crefname{enumi}{}{}

\usepackage[all]{xy}
\usepackage[usenames]{color}
\input{grcawol.sty}
\input{grcalcx.sty}
\input{grcalck.sty}
\renewcommand\theta{\vartheta}
\numberwithin{equation}{section}
\newcommand\HopfAlgB{\operatorname{\underline{HopfAlg}}(\mathcal B)}

\newcommand\CoalgB{\operatorname{\underline{Coalg}}(\mathcal B)}
\newcommand\Alg{\operatorname{\underline{Alg}}}
\newcommand\Coalg{\operatorname{\underline{Coalg}}}
\newcommand\Hopfend{\operatorname{End}}
\newcommand\Hopfaut{\operatorname{Aut}}
\newcommand\Hopfmor{\operatorname{Hom}}
\newcommand\Zenthom{\operatorname{Hom}^c}

\newtheorem{Lem}{Lemma}[section]
\newtheorem{Prop}[Lem]{Proposition}
\newtheorem{Cor}[Lem]{Corollary}
\newtheorem{Thm}[Lem]{Theorem}

\newtheorem*{intro-thm}{Theorem}
\theoremstyle{definition}
\newtheorem{Def}[Lem]{Definition}
\theoremstyle{remark}
\newtheorem{Rem}[Lem]{Remark}
\newtheorem{Rems}[Lem]{Remarks}
\newtheorem{Expl}[Lem]{Example}
\newcommand\ot{\otimes}
\newcommand\nob{I}
\DeclareMathOperator\id{\operatorname{id}}

\newcommand\inv{^{-1}}
\newcommand\Hom{\operatorname{Hom}}

\def\HMB#1.#2.#3.#4.{{^{#1}_{#3}\mathcal B^{#2}_{#4}}}

\renewcommand\epsilon\varepsilon

\newcommand\B{\mathcal B}
\newcommand\D{\mathcal D}

\newcommand\Z{\mathbb Z}
\newcommand\N{\mathbb N}

\newcommand\kk{\mathbbm k}
\newcommand\Ker{\operatorname{Ker}}
\newcommand\Coker{\operatorname{Coker}}
\newcommand\Img{\operatorname{Im}}
\newcommand\Coim{\operatorname{Coim}}
\def\BHM#1.#2.#3.#4.{{^{#1}_{#3}\mathcal B^{#2}_{#4}}}
\newcommand\rcofix[2]{{#1}{^{\operatorname{co} #2}}}
\newcommand\lcofix[2]{{{^{\operatorname{co} #2}}{#1}}}
\newcommand\lquot[2]{{#2}\backslash{#1}}

\newcommand\comm\curlyvee
\newcommand\cocomm\curlywedge
\newcommand\brd{\tau}
\newcommand\commute[2]{%
   \gbeg23
   \gcl1\gcl1\gnl
   \gbmp{#2}\gbmp{#1}\gnl
   \gmu\gend
   =
   \gbeg23
   \gbr\gnl
   \gbmp{#1}\gbmp{#2}\gnl
   \gmu\gend}
\newcommand\cocommute[2]{%
   \gbeg23
   \gcmu\gnl
   \gbmp{#2}\gbmp{#1}\gnl
   \gcl1\gcl1\gend
   =
   \gbeg23
   \gcmu\gnl
   \gbmp{#1}\gbmp{#2}\gnl
   \gbr\gend}
	
\newcommand{\du}[1]{\kk^{#1}} 
\newcommand{\tprod}[1]{\kk^{#1}\otimes\kk#1} 

\begin{document}
 \title[Tensor factorizations]{On tensor factorizations of Hopf algebras}
 \author{Marc Keilberg}
 \email{keilberg@usc.edu}
 \author{Peter Schauenburg}
 \email{peter.schauenburg@u-bourgogne.fr}
 \address{Institut de Math{\'e}matiques de Bourgogne, UMR 5584 du CNRS
\\
Universit{\'e} de Bourgogne\\
Facult{\'e} des Sciences Mirande\\
9 avenue Alain Savary\\
BP 47870 21078 Dijon Cedex\\
France
}
\thanks{Research partially supported through a FABER Grant by the \emph{Conseil régional de Bourgogne}}
\begin{abstract}
  We prove a variety results on tensor product factorizations of finite dimensional Hopf algebras (more generally Hopf algebras satisfying chain conditions in suitable braided categories). The results are analogs of well-known results on direct product factorizations of finite groups (or groups with chain conditions) such as Fitting's Lemma and the uniqueness of the Krull-Remak-Schmidt factorization. We analyze the notion of normal (and conormal) Hopf algebra endomorphisms, and the structure of endomorphisms and automorphisms of tensor products. The results are then applied to compute the automorphism group of the Drinfeld double of a finite group in the case where the group contains an abelian factor. (If it doesn't, the group can be calculated by results of the first author.)
\end{abstract}

\maketitle

%
%
\section*{Introduction}

The larger part of this paper is concerned with general results on Hopf algebras in braided categories generalizing well-known results from the theory of finite groups (or groups with chain conditions), such as Fitting's lemma, the Krull-Remak-Schmidt decomposition, and a description of endomorphisms and automorphisms of products of Hopf algebras. The last section deals with the description of the automorphism group of the Drinfeld double $\D(G)$ of a finite group $G$. This last problem was the starting point of our work.

 In the case that $G$ has no non-trivial abelian direct factors, a complete description of the automorphisms was given in \cite{K14}.  The case when $G$ has such an abelian factor was left open. We will write such a group as $G=C\times H$, where $H$ has no non-trivial abelian direct factors and $C$ is abelian.  In this case we naturally have that  $\D(G)\cong \D(C)\ot\D(H)$ is a tensor product of Hopf algebras.

 Thus, we are naturally led to analyze endomorphisms and automorphisms of a tensor product of two Hopf algebras. In \cite{BidCurMcC:ADPFG,Bid:ADPFG2} an analysis of the automorphisms of direct products of groups was provided. The basic idea is to describe such automorphisms by a matrix of morphisms between the factors. The machinery of normal group endomorphisms and Fitting's lemma then allows one to deduce conditions on the various morphisms from conditions on the factors. For example, when the two factors have no common direct factors, then the diagonal terms of the matrix have to be automorphisms. In \cref{sec:autom-tens-prod} we derive suitably analogous results for tensor product Hopf algebras. Before this can be done, however, we have to carry over to our Hopf algebraic setting some basic notions and classical results from group theory. In \cref{sec:comm-cocomm-morph} we develop the terminology of commuting morphisms (for groups these are just morphisms whose images commute) and dually of cocommuting morphisms, and in \cref{sec:normal-endomorphisms} the notions of normal and conormal Hopf endomorphisms. The analog of Fitting's lemma which will produce tensor product decompositions from binormal endomorphisms and thus, under suitable circumstances, common tensor factors from certain endomorphisms of tensor products, will be proved in \cref{sec-fitting}. An important application of Fitting's lemma in group theory is the uniqueness of the Krull-Remak-Schmidt decomposition, which we prove in \cref{sec:krull-remak-schmidt}. Extensions of the Krull-Remak-Schmidt decomposition were studied previously in \cite{Bur:CFPHA} for decompositions of semisimple Hopf algebras into simple semisimple tensor factors. By contrast our techniques make no use of semisimplicity but only of chain conditions. It is also worth noting that the Krull-Remak-Schmidt result shows that our results are specific to Hopf algebras and cannot be readily generalized to finite or even fusion tensor categories. In fact Müger \cite{Mue:SMC} gives an example where the factors in the decomposition of a fusion category into prime factors are not unique.

In fact the above results on the structure theory of finite dimensional Hopf algebras over a field $\kk$ will be developed in greater generality for Hopf algebras in braided abelian tensor categories that fulfill chain conditions on Hopf subalgebras and quotient Hopf algebras. Apart from the fact that the results will thus immediately apply to objects like super-Hopf algebras, for some purposes the categorical setting is simply very natural, since it allows treating mutually dual notions like normality and conormality or ascending and descending chain conditions on the same footing. If the braiding of the base category is not a symmetry, then some of our basic objects of study may be hard to come by: It is well-known that the tensor product of two Hopf algebras in a braided monoidal category can only be formed if the two factors are ``unbraided'', that is, the braiding between them behaves like a symmetry. On the other hand, some of our results imply that tensor product decompositions have to exist in certain situations. Thus these results also imply that the braiding has to be ``partially trivial''. For example, if non-nilpotent normal endomorphisms of a Hopf algebra exist, they have to be isomorphisms by Fitting's lemma unless the braiding is partially trivial. An automorphism between a tensor product of nonisomorphic Hopf algebras (necessarily ``unbraided'' between each other) has to induce automorphisms on the factors, unless the braiding is partially trivial on one of the factors.

\Cref{sec:endom-that-are} deals with some technical issues raised by our categorical framework. In preparation for Fitting's lemma we decompose a Hopf algebra with chain conditions, for which a Hopf algebra endomorphism is given, into a Radford biproduct (in the generalized braided version due to Bespalov and Drabant \cite{BesDra:HBMCMBC}). A technical result on (co)invariants under Hopf algebra endomorphisms has some bearing on the notions of epimorphisms and monomorphisms studied notably for infinite dimensional Hopf algebras in \cite{Chi:EMHA}.

In \cref{sec:appl-doubl-groups} we present the application of the general results on the structure of finite Hopf algebras and their automorphisms to the study of automorphisms of Drinfeld doubles of groups.  Letting $G=C\times H$ as before, taking the field to be the complex numbers, and defining $\widehat{H}$ to be the group of linear characters of $H$, then under the isomorphisms $\D(C)\cong\mathbb{C}(\widehat{C}\times C)$ and $\widehat{C}\times C\cong C^2$ the result can be stated as
\[\Hopfaut(\D(C\times H)) \cong \ \begin{pmatrix} \operatorname{Aut}(C^2)&\Hopfmor(\D(H),\mathbb{C}C^2)\\
	\operatorname{Hom}(C^2,\widehat{H}\times Z(H))&\Hopfaut(\D(H))\end{pmatrix}.\]
The only term not explicitly determined by \cite{K14} or standard methods for finite abelian groups is $\Hopfmor(\D(H),\mathbb{C}C^2)$.  In this case the morphisms can be described entirely in terms of group homomorphisms and central subgroups of $G$ satisfying certain relations \cite{ABM:CBPHA,K14}, so the description is not a significant problem.  In \cref{ex:dihedrals} we completely describe $\Hopfaut(\D(D_{2n}))$ where $D_{2n}$ is the dihedral group of order $2n$, for the case $n\equiv 2\bmod 4$ and $n>2$.  This is precisely when there is an isomorphism $D_{2n}\cong \Z_2\times D_n$.  From this we can easily provide a formula for the order of $\Hopfaut(\D(D_{2n}))$.  In particular we find that $\Hopfaut(\D(D_{12}))$ has order $1152=2^7 3^2$.

%
%
\section{Preliminaries and notation}
Throughout the paper, $\B$ is an abelian braided tensor category with braiding
$\brd$; we will assume that $\B$ is strict, backed up by the
well-known coherence theorems. Algebras,
coalgebras, bialgebras, Hopf algebras are in $\B$. All undecorated $\Hopfmor$s, $\Hopfend$s, etc. will be for morphisms of Hopf algebras or groups, as appropriate.  We will use
the following graphical notations to do computations in $\B$: The
braiding is
$$\brd_{VW}=\gbeg23\got1V\got1W\gnl\gbr\gnl\gob1W\gob1V\gend\qquad
  \text{ and }\qquad
  \brd_{VW}\inv=\gbeg23\got1W\got1V\gnl\gibr\gnl\gob1V\gob1W\gend.$$
We shall say that the objects $V$ and $W$ are \emph{unbraided} if
$\brd_{VW}=\brd_{WV}\inv$.

Multiplication and unit of an algebra $A$, and comultiplication and
counit of a coalgebra $C$ are
$$\nabla_A=\gbeg 23\got1A\got1A\gnl\gmu\gnl\gob2A\gend,
  \qquad
  \eta_A=\gbeg13\gnl\gu1\gnl\gob1A\gend,
  \qquad
  \Delta_C=\gbeg23\got2C\gnl\gcmu\gnl\gob1C\gob1C\gend,
  \qquad
  \epsilon_C=\gbeg13\got1C\gnl\gcu1\gnl\gend.
$$
The antipode of a Hopf algebra and, if it exists, its inverse are
$$S=\gbeg15\got1H\gnl\gcl1\gnl\gmp+\gnl\gcl1\gnl\gob1H\gend\qquad\text{ and
}\qquad
S\inv=\gbeg15\got1H\gnl\gcl1\gnl\gmp-\gnl\gcl1\gnl\gob1H\gend.$$

In order to have a straightforward notion of Hopf subalgebra and
quotient Hopf algebra of a given Hopf algebra, we shall assume that
tensor products in $\B$ are exact.

An object in $\B$ satisfies the ascending chain condition on
subobjects if and only if it satisfies the descending chain condition
on quotient objects, by which we understand the descending chain
condition on subobjects in the opposite category. For Hopf algebras we
will use the descending chain conditions on Hopf-subalgebras
and on quotient Hopf algebras.  This is done since Hopf algebras which are artinian as algebras are finite dimensional \cite{LiuZh07:AHFD}.  When a Hopf algebra satisifies the descending chain conditions on both Hopf-subalgebras and quotient Hopf algebras, we simply say that it satisfies both chain conditions.

If $f\colon H\to G$ is a Hopf algebra morphism, we define the right
and left $f$-coinvariant subobjects of $H$ as being the equalizers
\begin{gather*}
  \xymatrix{0\ar[r]&\rcofix Hf\ar[r]&H\ar@<0.5ex>[rr]^-{(H\ot
      f)\Delta}\ar@<-0.5ex>[rr]_-{H\ot\eta}&&H\ot G}\\
  \xymatrix{0\ar[r]&\lcofix Hf\ar[r]&H\ar@<0.5ex>[rr]^-{(f\ot
      H)\Delta}\ar@<-0.5ex>[rr]_-{\eta\ot H}&&G\ot H}
\end{gather*}
And dually, the left and right invariant quotients by coequalizers
\begin{gather*}
  \xymatrix{H\ot G\ar@<0.5ex>[rr]^-{\nabla(f\ot
      G)}\ar@<-0.5ex>[rr]_-{\epsilon\ot G}&&G\ar[r]&H\backslash G\ar[r]&0}\\
  \xymatrix{G\ot H\ar@<0.5ex>[rr]^-{\nabla(G\ot
      f)}\ar@<-0.5ex>[rr]_-{G\ot\epsilon}&&G\ar[r]&G/H
    \ar[r]&0}\end{gather*}

We note that the coinvariant subobjects are subalgebras of $H$, and
the invariant quotients are quotient coalgebras of $G$.

We will say a Hopf algebra is abelian if it is both commutative and cocommutative. In the category of vector spaces over a field $\kk$ of characteristic zero, such Hopf algebras are precisely group algebras of abelian groups, up to a separable field extension \cite[Theorem 2.3.1]{Mon:HAAR}.  We will say a Hopf algebra is non-abelian when it is not abelian.

%
%
\section{Commuting and Cocommuting morphisms}
\label{sec:comm-cocomm-morph}
In this section, we formulate an obvious commutation condition for
morphisms to an algebra (for ordinary algebras it just means that
elements in the respective images commute) and its dual, and we
collect equally obvious consequences that will be useful in later calculations. We note that for each and every fact on Hopf algebras in a braided category there is a dual fact. We will not always state, but still freely use the duals of our statements

Let $A$ be an algebra, $V,W\in\B$, and $f\colon V\to A,g\colon W\to A$
morphisms in $\B$. We say that $f$ and $g$ multiplication commute and
write $f\comm g$ if
$\nabla(g\ot f)=\nabla(f\ot g)\brd(=\nabla\brd(g\ot f))$, or
graphically\[\commute fg=\gbeg23\gbmp g\gbmp f\gnl\gbr\gnl\gmu\gend.\]

Dually, two morphisms $f\colon C\to V$ and $g\colon C\to W$ from a
coalgebra $C$ in $\B$ comultiplication commute, or cocommute for
short, and write $f\cocomm g$ if \[\cocommute fg=\gbeg23\gcmu\gnl\gbr\gnl\gbmp g\gbmp f\gend.\] 

We say that $f,g$ bicommute if both $f\comm g$ and $f\cocomm g$.

If $A$ and $B$ are algebras in $\B$, then the natural maps $f\colon
A\to A\ot B$ and $g\colon B\to A\ot B$ satisfy $f\comm g$, but they
only satisfy $g\comm f$ if $A$ and $B$ are unbraided: In fact 
\begin{align*}
  \gbeg25
  \got1A\got1B\gnl
  \gcl1\gcl1\gnl
\gbmp f\gbmp g\gnl
\gmu\gnl
\gob2{A\ot B}\gend
&=
\gbeg25
\got1A\got1B\gnl
\gcl3\gcl3\gnl\gnl\gnl
\gob1A\gob1B\gend
&\text{and}&&
  \gbeg25
  \got1B\got1A\gnl
  \gcl1\gcl1\gnl
\gbmp g\gbmp f\gnl
\gmu\gnl
\gob2{A\ot B}\gend
&=
\gbeg25
\got1B\got1A\gnl
\gcl1\gcl1\gnl
\gbr\gnl
\gcl1\gcl1\gnl
\gob1A\gob1B\gend.
\end{align*}

\begin{Lem}\label{usefulsimplecommutationfacts}
  Let $A$ be an algebra, $C$ a coalgebra, and $U,V,W,X,Y$ objects in $\B$.
  \begin{enumerate}
  \item Let $f\colon U\to A$, $g\colon V\to A$ and
    $h\colon W\to A$. 
    \begin{enumerate}
    \item If $f\comm g$ and $f\comm h$, then $f\comm(\nabla(g\ot h))$.
    \item If $f\comm h$ and $g\comm h$, then $(\nabla(f\ot g))\comm
      h$.
    \item If $f\comm g$, then $fa\comm gb$ for any $a\colon X\to U$
      and $b\colon Y\to V$.\label{compositescommute}
    \end{enumerate}
    \item Let $f,g,h\colon C\to A$.
      \begin{enumerate}
      \item If $f\comm g$ and $f\comm h$ then $f\comm(g*h)$.
      \item If $f\comm h$ and $g\comm h$ then $(f*g)\comm h$.
       \item If $f\comm g$ and $g\cocomm f$, then $f*g=g*f$.
      \end{enumerate}
    \item Let $f,g\colon C\to A$.
      \begin{enumerate}
      \item If $C$ is a bialgebra, $f,g$ are algebra morphisms, and
        $f\comm g$, then $f*g$ is an algebra morphism.
      \item If $A,C$ are bialgebras, $f,g$ are bialgebra morphisms,
        $f\comm g$ and $f\cocomm g$, then $f*g$ is a bialgebra morphism.
			\item If $A$ is a bialgebra, $C$ a Hopf algebra, and $f,g$ are unital coalgebra morphisms, then $f\cocomm g$ $\iff$ $f*g$ is a coalgebra morphism.
      \end{enumerate}
  \end{enumerate}
\end{Lem}

Note that $f\comm g$ is not necessarily
equivalent to $g\comm f$ in the braided setting. The first part of the following result says, however, that the two properties are equivalent for Hopf algebras with sufficiently well-behaved antipodes. On the other hand, the second part says that if both properties are fulfilled then either the braiding is close to being a symmetry, or the morphisms are close to being trivial.
\begin{Prop}
  Let $H,K,$ and $A$ be Hopf algebras, and $f\colon H\to A$, $g\colon K\to A$
  Hopf algebra morphisms.
\begin{enumerate}
\item If $f\comm g$, and if the antipode of $A$ is a monomorphism   or the antipodes of $H$ and $K$ are epimorphisms, then $g\comm f$.
  \item If $f\comm g$ and $g\comm f$, then
  \begin{equation*}
    \gbeg25
    \got1H\got1K\gnl
    \gcl1\gcl1\gnl
    \gbmp f\gbmp g\gnl
    \gbr\gnl
    \got1A\got1A\gend
    =
    \gbeg25
    \got1H\got1K\gnl
    \gcl1\gcl1\gnl
    \gbmp f\gbmp g\gnl
    \gibr\gnl
    \got1A\got1A\gend
  \end{equation*}
\end{enumerate}
\end{Prop}
\begin{proof}
For the first claim, we calculate
\begin{equation*}
  \gbeg26
  \got1H\got1K\gnl
  \gcl1\gcl1\gnl
  \gmp+\gmp+\gnl
  \gbmp f\gbmp g\gnl
  \gmu\gnl
  \gob2A\gend
  =
  \gbeg26
  \got1H\got1K\gnl
  \gcl1\gcl1\gnl
  \gbmp f\gbmp g\gnl
  \gmp+\gmp+\gnl
  \gmu\gnl
  \gnl
  \gob2A\gend
  =
  \gbeg37
  \got1H\gvac1\got1K\gnl
  \gcl1\gvac1\gcl1\gnl
  \gbmp f\gvac1\gbmp g\gnl
  \gibbr31\gnl
  \gwmu3\gnl
  \gvac1\gmp+\gnl
  \gob3A\gend
  =
  \gbeg36
  \got1H\gvac1\got1K\gnl
  \gcl1\gvac1\gcl1\gnl
  \gbmp f\gvac1\gbmp g\gnl
  \gwmu3\gnl
  \gvac1\gmp+\gnl
  \gob3A\gend
  =
  \gbeg27
  \got1H\got1K\gnl
  \gcl1\gcl1\gnl
  \gbmp f\gbmp g\gnl
  \gmp+\gmp+\gnl
  \gbr\gnl
  \gmu\gnl
  \gob2A\gend  
  =
  \gbeg27
  \got1H\got1K\gnl
  \gcl1\gcl1\gnl
  \gmp+\gmp+\gnl
  \gbmp f\gbmp g\gnl
  \gbr\gnl
  \gmu\gnl
  \gob2A\gend,
\end{equation*}
which implies $g\comm f$ if the antipodes of $H$ and $K$ are epimorphisms. A similar argument shows the same if the antipode of $A$ is a monomorphism.

We now turn to the second claim.
  
  First, we note that
  \begin{multline*}
    \gbeg46
    \got2H\got2K\gnl
    \gcmu\gcmu\gnl
    \gcl1\gibr\gcl1\gnl
    \gbmp f\gbmp g\gbmp f\gbmp g\gnl
    \gmu\gmu\gnl
    \gob2A\gob2A\gend
    =
    \gbeg47
    \got2H\got2K\gnl
    \gcmu\gcmu\gnl
    \gcl1\gibr\gcl1\gnl
    \gbmp f\gbmp g\gbmp f\gbmp g\gnl
    \gibr\gibr\gnl
    \gmu\gmu\gnl
    \gob2A\gob2A\gend
    =
    \gbeg58
    \gvac1\got1H\gvac1\got1K\gnl
    \gvac1\gcl1\gvac1\gcl1\gnl
    \gvac1\gbmp f\gvac1\gbmp g\gnl
    \gwcmh324\gvac{-1}\gwcmh324\gnl
    \gcn1122\gibbrh2124\gcn0122\gnl
    \gibbrh2124\gibbrh2124\gnl
    \gwmuh324\gvac{-1}\gwmuh324\gnl
    \gvac1\gob1A\gvac1\gob1A\gend
    =
    \gbeg59
    \gvac1\got1H\gvac1\got1K\gnl
    \gvac1\gcl1\gvac1\gcl1\gnl
    \gvac1\gbmp f\gvac1\gbmp g\gnl
    \gibbrh4237\gnl
    \gnl
    \gwcmh324\gvac{-1}\gwcmh324\gnl
    \gcn1122\gbbrh2124\gcn0122\gnl
    \gwmuh324\gvac{-1}\gwmuh324\gnl
    \gvac1\gob1A\gvac1\gob1A\gend
    =\\\\=
    \gbeg27
    \got1H\got1K\gnl
    \gcl1\gcl1\gnl
    \gbmp f\gbmp g\gnl
    \gibr\gnl
    \gmu\gnl
    \gcmu\gnl
    \gob1A\gob1A\gend
    =
    \gbeg26
    \got1H\got1K\gnl
    \gcl1\gcl1\gnl
    \gbmp f\gbmp g\gnl
    \gmu\gnl
    \gcmu\gnl
    \gob1A\gob1A\gend
    =
    \gbeg57
    \gvac1\got1H\gvac1\got1K\gnl
    \gvac1\gcl1\gvac1\gcl1\gnl
    \gvac1\gbmp f\gvac1\gbmp g\gnl
    \gwcmh324\gvac{-1}\gwcmh324\gnl
    \gcn1122\gbbrh2124\gcn0122\gnl
    \gwmuh324\gvac{-1}\gwmuh324\gnl
    \gvac1\gob1A\gvac1\gob1A\gend
    =
    \gbeg46
    \got2H\got2H\gnl
    \gcmu\gcmu\gnl
    \gcl1\gbr\gcl1\gnl
    \gbmp f\gbmp g\gbmp f\gbmp g\gnl
    \gmu\gmu\gnl
    \gob2A\gob2A\gend
\end{multline*}
  In other words 
  \begin{equation*}
    \gbeg46
    \got2H\got2H\gnl
    \gcmu\gcmu\gnl
    \gcl1\gdnot X\glmptb\grmptb\gcl1\gnl
    \gbmp f\gbmp g\gbmp f\gbmp g\gnl
    \gmu\gmu\gnl
    \gob2A\gob2A\gend
  \end{equation*}
  does not depend on the choice of
  $X\in\left\{\gbeg21\gbr\gend,\gbeg21\gibr\gend\right\}$. But since
  $f\circ S$ and $g\circ S$ are convolution inverse to $f$ and $g$, respectively, we have
  \begin{equation*}
    \gbeg68
    \got3H\got3K\gnl
    \gwcmh314\gwcmh325\gnl
    \gcl1\gcmu\gcmu\gcl1\gnl
    \gmp+\gcl1\gdnot X\glmptb\grmptb\gcl1\gmp+\gnl
    \gbmp f\gbmp f\gbmp g\gbmp f\gbmp g\gbmp g\gnl
    \gcl1\gmu\gmu\gcl1\gnl
    \gwmuh314\gwmuh325\gnl
    \gob3A\gob3A\gend
    =
    \gbeg68
    \got3H\got3K\gnl
    \gwcmh325\gwcmh314\gnl
    \gcmu\gcl1\gcl1\gcmu\gnl
    \gmp+\gcl1\gdnot X\glmptb\grmptb\gcl1\gmp+\gnl
    \gbmp f\gbmp f\gbmp g\gbmp f\gbmp g\gbmp g\gnl
    \gmu\gcl1\gcl1\gmu\gnl
    \gwmuh325\gwmuh314\gnl
    \gob3A\gob3A\gend
    =
    \gbeg46
    \got2H\got2K\gnl
    \gcmu\gcmu\gnl
    \gcu1\gdnot X\glmptb\grmptb\gcu1\gnl
    \gu1\gbmp g\gbmp f\gu1\gnl
    \gmu\gmu\gnl
    \gob2A\gob2A\gend
    =
    \gbeg26
    \got1H\got1K\gnl
    \gcl1\gcl1\gnl
    \gdnot X\glmptb\grmptb\gnl
    \gbmp g\gbmp f\gnl
    \gcl1\gcl1\gnl
    \gob1A\gob1A\gend
  \end{equation*}
  That the latter expression does not depend on the choice of $X$ is
  the claim.
\end{proof}

As special cases one recovers two known facts that show how badly usual Hopf algebra constructions behave in a ``truly braided'' tensor category: A Hopf algebra cannot be commutative (or cocommutative) as a (co)algebra in $\mathcal B$ unless the braiding on the Hopf algebra is an involution \cite{Sch:BHABC}, and the tensor product of two Hopf algebras cannot be a Hopf algebra unless the two factors are unbraided.

%
%
\section{Normal endomorphisms}
\label{sec:normal-endomorphisms}

Recall that the left adjoint action and the left coadjoint coaction of
a Hopf algebra $H$ on itself are
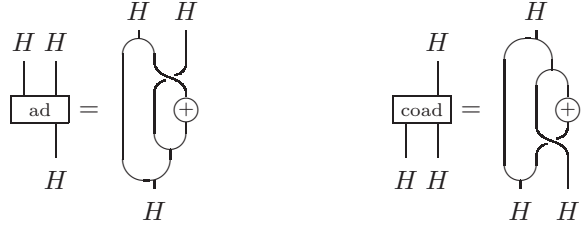
\begin{align*}
  \gbeg25
  \got1H\got1H\gnl
  \gcl1\gcl1\gnl
  \gdnot{\operatorname{ad}}\glmpt\grmptb\gnl
  \gvac1\gcl1\gnl
  \gvac1\gob1H\gend
  &=
  \gbeg37
  \got2H\got1H\gnl
  \gcmu\gcl1\gnl
  \gcl3\gbr\gnl
  \gvac1\gcl1\gmp+\gnl
  \gvac1\gmu\gnl
  \gwmuh314\gnl
  \gob3H\gend
  &
  \gbeg25
  \gvac1\got1H\gnl
  \gvac1\gcl1\gnl
  \gdnot{\operatorname{coad}}\glmpb\grmptb\gnl
  \gcl1\gcl1\gnl
  \gob1H\gob1H\gend
  &=
  \gbeg37
  \got3H\gnl
  \gwcmh314\gnl
  \gcl3\gcmu\gnl
  \gvac1\gcl1\gmp+\gnl
  \gvac1\gbr\gnl
  \gmu\gcl1\gnl
  \gob2H\gob1H\gend
\end{align*}

We note that the adjoint action is characterized by a twisted commutativity condition:
\begin{equation}
  \label{eq:3}
  \gbeg36
  \got 2H\got1H\gnl
  \gcmu\gcl1\gnl
  \gcl1\gbr\gnl
  \gdnot{\operatorname{ad}}\glmpt\grmptb\gcl1\gnl
  \gvac1\gmu\gnl
  \gvac1\gob2H\gend
  =
  \gbeg25
  \got1H\got1H\gnl
  \gcl1\gcl1\gnl
  \gmu\gnl
  \gcn2122\gnl
  \gob2H\gend
\end{equation}

\begin{Def}
  Let $f\colon H\to H$ be a morphism in $\B$, with $H$ a Hopf algebra.
  \begin{enumerate}
  \item $f$ is normal if it is left $H$-linear with respect to the
    adjoint action.
  \item $f$ is conormal if it is left $H$-colinear with respect to the
    coadjoint coaction.
	\item $f$ is binormal if it is both normal and conormal.
  \end{enumerate}
\end{Def}
For group algebras considered in the category of $\mathbb{C}$-vector spaces, the definition of a normal morphism agrees with the one used in group theory \cite{Rot:Book}.  Since group algebras are cocommutative, every group endomorphism is trivially conormal.  We will be primarily concerned with normal algebra morphisms, conormal coalgebra morphisms, and binormal bialgebra morphisms.
\begin{Lem}
  Let $f\colon H\to H$ be an endomorphism of the Hopf algebra $H$.
  \begin{enumerate}
  \item The following are equivalent:
    \begin{enumerate}
    \item $f$ is normal.\label{equiv:normal}
    \item $f\comm((fS)*\id_H)$.\label{equiv:comm}
    \item $(fS)*\id_H$ is an algebra morphism.\label{equiv:algmor}
    \end{enumerate}
	\item The following are equivalent:
		\begin{enumerate}
    \item $f$ is binormal.
    \item $f\cocomm((fS)*\id_H)$ and $f\comm ((fS)*\id_H)$.
    \item $(fS)*\id_H$ is a bialgebra morphism
    \end{enumerate}
  \end{enumerate}
\end{Lem}
\begin{proof}
  We only show the first part. For the equivalence of
  \cref{equiv:comm} and \cref{equiv:algmor} we apply the bijection
  \begin{equation*}
    \B(H\ot H,H)\ni T\mapsto 
    \gbeg35
    \got1H\got2H\gnl
    \gcl1\gcmu\gnl
    \gdnot T\glmpt\grmptb\gmp+\gnl
    \gvac1\gmu\gnl
    \gvac1\gob2H\gend
    \in\B(H\ot H,H)
  \end{equation*}
to the two sides of the equation expressing multiplicativity of
$g:=fS*\id$. We get
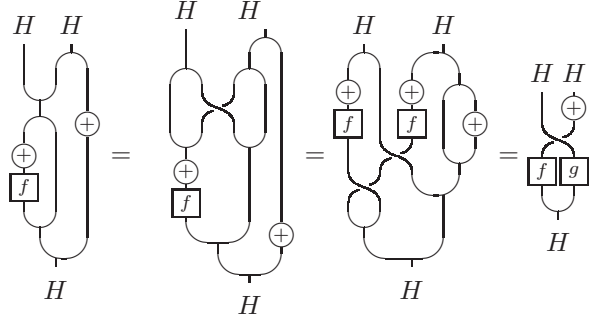
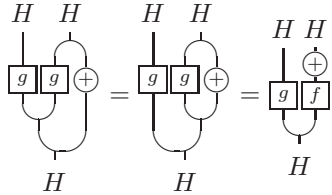
\begin{equation*}
  \gbeg39
  \got1H\got2H\gnl
  \gcl1\gcmu\gnl
  \gmu\gcl1\gnl
  \gcmu\gmp+\gnl
  \gmp+\gcl2\gcl3\gnl
  \gbmp f\gnl
  \gmu\gnl
  \gwmuh325\gnl
  \gob3H\gend
  =
  \gbeg5{10}
  \got3H\got1H\gnl
  \gvac1\gcl1\gvac1\gcmu\gnl
  \gwcmh324\gvac{-1}\gwcmh324\gvac{-1}\gcn1511\gnl
  \gcn1122\gbbrh2124\gcn0122\gnl
  \gwmuh324\gvac{-1}\gwmuh324\gnl
  \gvac1\gmp+\gvac1\gcl2\gnl
  \gvac1\gbmp f\gnl
  \gvac1\gwmu3\gmp+\gnl
  \gvac2\gwmu3\gnl
  \gvac2\gob3H\gend
  =
  \gbeg59
  \got2H\got3H\gnl
  \gcmu\gwcmh314\gnl
  \gmp+\gcl2\gmp+\gcmu\gnl
  \gbmp f\gvac1\gbmp f\gcl1\gmp+\gnl
  \gcl1\gbr\gmu\gnl
  \gbr\gwmuh314\gnl
  \gmu\gvac1\gcl1\gnl
  \gwmuh527\gnl
  \gob5H
  \gend
  =
  \gbeg26
  \got1H\got1H\gnl
  \gcl1\gmp+\gnl
  \gbr\gnl
  \gbmp f\gbmp g\gnl
  \gmu\gnl
  \gob2H\gend 
\end{equation*}
and
\begin{equation*}
  \gbeg36
  \got1H\got2H\gnl
  \gcl1\gcmu\gnl
  \gbmp g\gbmp g\gmp+\gnl
  \gmu\gcl1\gnl
  \gwmuh325\gnl
  \gob3H\gend
  =
  \gbeg36
  \got1H\got2H\gnl
  \gcl1\gcmu\gnl
  \gbmp g\gbmp g\gmp+\gnl
  \gcl1\gmu\gnl
  \gwmuh314\gnl
  \gob3H\gend
  =
  \gbeg25
  \got1H\got1H\gnl
  \gcl1\gmp+\gnl
  \gbmp g\gbmp f\gnl
  \gmu\gnl
  \gob2H\gend
\end{equation*}
that is, the two sides of \cref{equiv:comm}, up to composition with
the isomorphism $H\ot S$.

For the equivalence of \cref{equiv:normal} and \cref{equiv:comm}, we
apply the bijection
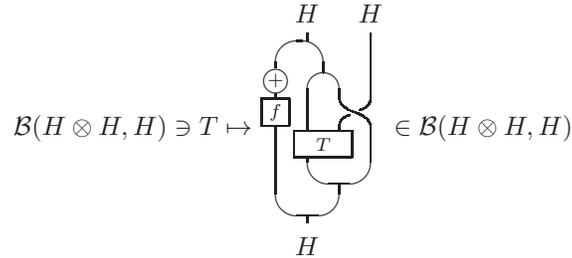
\begin{equation*}
  \B(H\ot H,H)\ni T\mapsto
  \gbeg48
  \got3H\got1H\gnl
  \gwcmh314\gcl2\gnl
  \gmp+\gcmu\gnl
  \gbmp f\gcl1\gbr\gnl
  \gcl2\gdnot T\glmptb\grmpt\gcl1\gnl
  \gvac1\gwmu3\gnl
  \gwmu3\gnl
  \gob3H\gend
  \in\B(H\ot H,H)
\end{equation*}
to the two sides of \cref{equiv:normal} to get
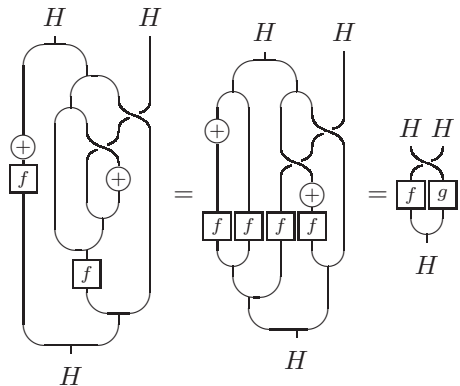
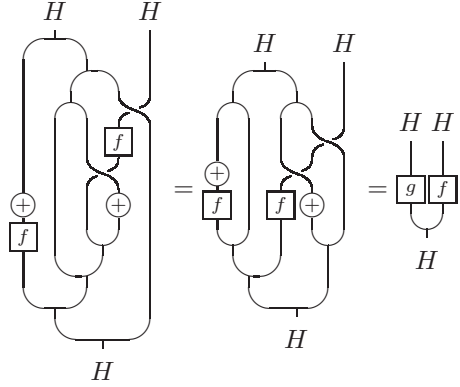
\begin{equation*}
  \gbeg5{12}
  \got3H\gvac1\got1H\gnl
  \gwcm3\gvac1\gcl2\gnl
  \gcl2\gwcmh325\gnl
  \gvac1\gcmu\gbr\gnl
  \gmp+\gcl3\gbr\gcl5\gnl
  \gbmp f\gvac1\gcl1\gmp+\gnl
  \gcl4\gvac1\gmu\gnl
  \gvac1\gwmuh314\gnl
  \gvac2\gbmp f\gnl
  \gvac2\gwmu3\gnl
  \gwmu4\gnl
  \gob4H\gend
  =
  \gbeg5{11}
  \got4H\got1H\gnl
  \gwcmh426\gcl2\gnl
  \gcmu\gcmu\gnl
  \gmp+\gcl3\gcl1\gbr\gnl
  \gcl2\gvac1\gbr\gcl3\gnl
  \gvac2\gcl1\gmp+\gnl
  \gbmp f\gbmp f\gbmp f\gbmp f\gnl
  \gmu\gcl1\gmu\gnl
  \gwmuh325\gcn2122\gnl
  \gvac1\gwmuh416\gnl
  \gvac1\gob4H\gend
  =
  \gbeg25
  \got1H\got1H\gnl
  \gbr\gnl
  \gbmp f\gbmp g\gnl
  \gmu\gnl
  \gob2H\gend
\end{equation*}
and
\begin{equation*}
  \gbeg5{12}
  \got3H\gvac1\got1H\gnl
  \gwcm3\gvac1\gcl2\gnl
  \gcl4\gwcmh325\gnl
  \gvac1\gcmu\gbr\gnl
  \gvac1\gcl4\gcl1\gbmp f\gcl6\gnl
  \gvac2\gbr\gnl
  \gmp+\gvac1\gcl1\gmp+\gnl
  \gbmp f\gvac1\gmu\gnl
  \gcl1\gwmuh314\gnl
  \gwmu3\gnl
  \gvac1\gwmu4\gnl
  \gvac1\gob4H\gend
  =
  \gbeg5{10}
  \got4H\got1H\gnl
  \gwcmh426\gcl2\gnl
  \gcmu\gcmu\gnl
  \gcl1\gcl3\gcl1\gbr\gnl
  \gmp+\gvac1\gbr\gcl2\gnl
  \gbmp f\gvac1\gbmp f\gmp+\gnl
  \gmu\gcl1\gmu\gnl
  \gwmuh325\gcn2122\gnl
  \gvac1\gwmuh416\gnl
  \gvac1\gob4H
  \gend
  =
  \gbeg25
  \got1H\got1H\gnl
  \gcl1\gcl1\gnl
  \gbmp g\gbmp f\gnl
  \gmu\gnl
  \gob2H\gend
\end{equation*}
which are the two sides of \cref{equiv:normal}.
\end{proof}

%
%
\section{Epic or monic endomorphisms}
\label{sec:endom-that-are}

We recall Radford's theorem on Hopf algebras with a projection
\cite{Rad:SHAP}, which was generalized to a categorical setting even
more general than the one in the present paper by Bespalov and Drabant
\cite{BesDra:HBMCMBC}:
\begin{Thm}
  Let $H$ be a Hopf algebra, and $\pi$ an idempotent Hopf algebra
  endomorphism of $H$. Then $H\cong \Img(\pi)\ot \Img(p)$, where
  $p=(\pi\circ S)*\id_H$ is an idempotent endomorphism of the object
  $H$ in $\B$ (but not necessarily a Hopf endomorphism). $B:=\Img(p)$
  is a subalgebra and a quotient coalgebra of $H$. The algebra
  structure of $\Img(\pi)\ot B$ is a semidirect product with respect to
  a certain action of $K=\Img(\pi)$ on $B$, and the coalgebra
  structure is the cosemidirect product with respect to a certain
  coaction.

  Moreover $\Img(p)\cong\lcofix H\pi\cong \lquot H\pi$.
\end{Thm}
\begin{proof}
  Only the last statement is not in \cite{BesDra:HBMCMBC}, who avoid using
  coinvariant subobjects altogether to generalize \cite{Rad:SHAP} to
  categories that might not have equalizers. We check the first
  isomorphism: We find
  \begin{equation*}
    \gbeg37
    \got3H\gnl
    \gvac1\gcl1\gnl
    \gvac1\gbmp p\gnl
    \gwcm3\gnl
    \gbmp\pi\gvac1\gcl2\gnl
    \gcl1\gnl
    \gob1H\gvac1\gob1H\gend
    =
    \gbeg29
    \got2H\gnl
    \gcmu\gnl
    \gmp+\gcl2\gnl
    \gbmp\pi\gnl
    \gmu\gnl
    \gcmu\gnl
    \gbmp\pi\gcl2\gnl
    \gcl1\gnl
    \gob1H\gob1H\gend
    =
    \gbeg5{10}
    \got5H\gnl
    \gvac1\gwcm3\gnl
    \gvac1\gmp+\gvac1\gcl2\gnl
    \gvac1\gbmp\pi\gnl
    \gwcmh324\gvac{-1}\gwcmh324\gnl
    \gcn2122\gvac{-1}\gbbrh2124\gcn2122\gnl
    \gwmuh324\gvac{-1}\gwmuh324\gnl
    \gvac1\gbmp\pi\gvac1\gcl2\gnl
    \gvac1\gcl1\gnl
    \gvac1\gob1H\gvac1\gob1H
    \gend
    =
    \gbeg49
    \got4H\gnl
    \gwcmh426\gnl
    \gcmu\gcmu\gnl
    \gbr\gcl2\gcl4\gnl
    \gmp+\gmp+\gnl
    \gbmp\pi\gbr\gnl
    \gbmp\pi\gbmp\pi\gbmp\pi\gnl
    \gmu\gmu\gnl
    \gob2H\gob2H\gend
    =
    \gbeg4{10}
    \got3H\gnl
    \gwcm3\gnl
    \gcl2\gwcmh325\gnl
    \gvac1\gcmu\gcl5\gnl
    \gbr\gcl2\gnl
    \gmp+\gmp+\gnl
    \gcl1\gbr\gnl
    \gbmp\pi\gbmp\pi\gbmp\pi\gnl
    \gmu\gmu\gnl
    \gob2H\gob2H\gend
    =
    \gbeg36
    \gvac1\got2H\gnl
    \gvac1\gcmu\gnl
    \gvac1\gmp+\gcl2\gnl
    \gu1\gbmp\pi\gnl
    \gcl1\gmu\gnl
    \gob1H\gob2H\gend
  \end{equation*}
  and if some morphism $t\colon T\to H$ satisfies $(\pi\ot\id_H)\Delta
  t=\eta\ot t$, then
  \begin{equation*}
    pt
    =
    \gbeg38
    \got3T\gnl
    \gvac1\gcl1\gnl
    \gvac1\gbmp t\gnl
    \gwcm3\gnl
    \gbmp\pi\gvac1\gcl2\gnl
    \gmp+\gnl
    \gwmu3\gnl
    \gob3H\gend
    =
    \gbeg25
    \gvac1\got1T\gnl
    \gvac1\gcl1\gnl
    \gu1\gbmp t\gnl
    \gmu\gnl
    \gob2H\gend
    =t
  \end{equation*}
\end{proof}

\begin{Prop}\label{Radford}
  Let $H$ be a Hopf algebra, and $f$ a Hopf algebra endomorphism of
  $H$.
  
  Assume that $H$ satisfies both chain conditions.

  Then there is $n\in\N$ such that $H\cong \Img(f^n)\ot \lcofix
  H{f^n}$ is a Radford biproduct.
\end{Prop}
\begin{proof}
  Consider the epi-mono factorization $f=(H\xrightarrow eB\xrightarrow
  mH)$, where we identify $B=\Img(f)=\Coim(f)$. Then the endomorphism
  $t=em$ of $B$ satisfies $mt=fm$ and $te=ef$. The chain conditions on
  $H$ imply that the ascending chain of the kernels of $f^n$ and the
  descending chain of the images, hence the ordered chain of quotient
  objects formed by the cokernels of $f^n$ stablilize. Then, replacing
  $f$ by a suitable power $f^n$, we can assume that $t$ is an
  isomorphism. Then $\pi=mt\inv e$ is an idempotent endomorphism of
  $H$, since $\pi^2=mt\inv emt\inv e=mt\inv tt\inv e=mt\inv e=\pi$.

  Thus $H\cong \Img(\pi)\ot \lcofix H\pi$ is a Radford
  biproduct. Moreover, $\Img(\pi)=\Img(f)$, and $\lcofix
  H\pi=\lcofix Hf$.
\end{proof}
\begin{Prop}
  Let $H$ be a Hopf algebra in $\B$ that satisfies both chain conditions, and
  $f$ a Hopf algebra endomorphism of $H$.
  \begin{enumerate}
  \item If the left or right $f$-coinvariants of $H$ are trivial, then
    $f$ is a monomorphism in $\B$.
  \item If the left or right $f$-invariant quotient of $H$ is trivial,
    then $f$ is an epimorphism in $\B$.
  \end{enumerate}
\end{Prop}
\begin{proof}
  We prove the first part. By \cref{Radford}, $H\cong \Img(f^n)\ot
  \lcofix H{f^n}$ is a Radford biproduct for some $n$. If $\lcofix H{f^n}$ were trivial without $f^n$ being monic, it
  would follow that $H$ is isomorphic to a proper quotient of itself,
  contradicting the chain conditions. Now assume for some $m>1$ that
  $\lcofix H{f^m}$ is nontrivial. Let $j\colon\lcofix H{f^m}\to H$ be
  the inclusion. By exactness of tensor products in $\B$, we have an
  equalizer
  \begin{equation*}
    \xymatrix{0\ar[r]&\lcofix Hf\ot H\ar[r]&H\ar@<0.5ex>[rr]^-{(f\ot
        H)\Delta\ot H}\ar@<-0.5ex>[rr]_-{\eta\ot H\ot H}&&H\ot H\ot H}
  \end{equation*}
  and by the calculation
  \begin{multline*}
    ((f^{m-1}\ot H)\Delta\ot H)(f\ot H)\Delta j
    =(f^m\ot f\ot H)(\Delta\ot H)\Delta j
    \\=(f^m\ot (f\ot H)\Delta)\Delta j
    =\eta\ot (f\ot H)\Delta j
  \end{multline*}
  we see that $(f\ot H)\Delta j$ factors through this equalizer. We
  conclude that if $(f\ot H)\Delta j$ were not trivial, then it would
  follow that $\lcofix H{f^{m-1}}\ot H$ is not $\nob\ot H$, which
  implies $\lcofix H{f^{m-1}}$ is nontrivial. We can conclude by
  induction that $\lcofix Hf$ is nontrivial after all.
\end{proof}
\begin{Rems}
  Let $f\colon H\to G$ be a Hopf algebra homomorphism in
  $\B$.
  \begin{enumerate}
  \item Clearly, if $f$ is a monomorphism in $\B$, then it is a
  monomorphism in $\HopfAlgB$.
  \item If $f$ has trivial left or right coinvariants, then $f$ is a
    monomorphism in $\CoalgB$.
  \item If $f$ is normal, and a monomorphism in $\HopfAlgB$, then $f$
    has trivial left and right coinvariants.
  \end{enumerate}
  Thus the preceding result shows that normal endomorphisms of a Hopf
  algebra in $\B$ are monic (epic) if and only if they are so
  considered as morphisms in $\B$. 
\end{Rems}

\begin{proof}
  If $C$ is a coalgebra and $g,h\colon C\to H$ are coalgebra morphisms
  with $fg=fh$, then 
  \begin{multline*}
    \gbeg29
    \got2C\gnl
    \gcmu\gnl
    \gbmp g\gbmp h\gnl
    \gcl1\gmp+\gnl
    \gmu\gnl
    \gcmu\gnl
    \gcl2\gbmp f\gnl
    \gvac1\gcl1\gnl
    \gob1H\gob1G\gend
    =
    \gbeg5{10}
    \got5C\gnl
    \gvac1\gwcm3\gnl
    \gvac1\gbmp g\gvac1\gbmp h\gnl
    \gvac1\gcl1\gvac1\gmp+\gnl
    \gwcmh324\gvac{-1}\gwcmh324\gnl
    \gcn1122\gbbrh2124\gcn1122\gnl
    \gwmuh324\gvac{-1}\gwmuh324\gnl
    \gvac1\gcl2\gvac1\gbmp f\gnl
    \gvac3\gcl1\gnl
    \gvac1\gob1H\gvac1\gob1G\gend
    =
    \gbeg49
    \got4C\gnl
    \gwcmh426\gnl
    \gcmu\gcmu\gnl
    \gcl3\gcl1\gbr\gnl
    \gvac1\gbr\gbmp h\gnl
    \gvac1\gmp+\gbmp g\gmp +\gnl
    \gbmp g\gbmp h\gbmp f\gbmp f\gnl
    \gmu\gmu\gnl
    \gob2H\gob2G\gend
    =
    \gbeg49
    \got4C\gnl
    \gwcmh426\gnl
    \gcmu\gcmu\gnl
    \gcl3\gcl1\gbr\gnl
    \gvac1\gbr\gbmp g\gnl
    \gvac1\gmp+\gbmp g\gmp +\gnl
    \gbmp g\gbmp h\gbmp f\gbmp f\gnl
    \gmu\gmu\gnl
    \gob2H\gob2G\gend
    =\\\\=
    \gbeg4{11}
    \got4C\gnl
    \gwcmh416\gnl
    \gcl6\gwcmh425\gnl
    \gvac1\gcmu\gcl2\gnl
    \gvac1\gbmp g\gbmp g\gnl
    \gvac1\gcl1\gbr\gnl
    \gvac1\gbr\gmp+\gnl
    \gvac1\gbmp h\gbmp f\gbmp f\gnl
    \gbmp g\gmp+\gmu\gnl
    \gmu\gcn2122\gnl
    \gob2H\gob2G
    \gend
    =
    \gbeg36
    \got2C\gnl
    \gcmu\gnl
    \gbmp g\gbmp h\gnl
    \gcl1\gmp+\gu1\gnl
    \gmu\gcl1\gnl
    \gob2H\gob1G\gend
  \end{multline*}
  and thus, if $\rcofix Hf$ is trivial, $g*Sh=\eta\epsilon$, whence
  $g=h$. If $f$ is normal, then the coinvariants are a Hopf subalgebra.
\end{proof}

\begin{Rem}
	In general it is false that monic is equivalent to trivial coinvariants, or that epic is equivalent to trivial invariants.  In finite dimensions these concepts agree by the Nichols-Zoeller theorem \cite{NicZoe:HAFT,Scha:NZTHACYDM}.  In infinite dimensions, however, counterexamples are known \cite{Chi:EMHA}.
\end{Rem}

\begin{Lem}
  Let $H$ be a Hopf algebra in $\B$ that satisfies both chain conditions. Assume further that the braiding $\brd_{HH}$ has finite order. Then the antipode of $H$ is an automorphism in $\B$.
\end{Lem}
\begin{proof}
  Depict the iterates of the antipode by
  \begin{equation*}
    S^m=\gbeg13\gcl1\gnl\gmp m\gnl\gcl1\gend
  \end{equation*}
  One has
  \begin{equation*}
    \gbeg24
    \gcl1\gcl1\gnl
    \gmp m\gmp m\gnl
    \gdnot{\tau^m}\glmptb\grmptb\gnl
    \gmu\gend
    =
    \gbeg33
    \gwmuh324\gnl
    \gvac1\gmp m\gnl
    \gvac1\gcl1\gend
  \end{equation*}
  Using this, we can show inductively that the coinvariants of $H$ under an iterate of the antipode are trivial as follows: Let $t\colon T\to H$ be a morphism factoring through $\lcofix H{S^{2n}}$, i.~e.\ $(S^{2n}\ot H)\Delta t=\eta\ot t$. We will show that $(S^m\ot H)\Delta t=\eta\ot t$ for any $m$, whence (taking $m=0$) $t=\eta\epsilon t$.

Assume $(S^{m+1}\ot H)\Delta t=\eta\ot t$, or pictorially
\begin{equation*}
  \gbeg38
  \got3T\gnl
  \gvac1\gcl1\gnl
  \gvac 1\gbmp t\gnl
  \gwcm 3\gnl
  \gmp +\gvac1\gcl3\gnl
  \gmp m\gnl
  \gcl1\gnl
  \gob1H\gvac1\gob1H\gend
  =
  \gbeg25
  \gvac1\got1T\gnl
  \gvac1\gcl1\gnl
  \gu1\gbmp t\gnl
  \gcl1\gcl1\gnl
  \gob1H\gob1H\gend
\end{equation*}
Then
\begin{equation*}
  \gbeg25
  \gvac1\got1T\gnl
  \gvac1\gcl1\gnl
  \gu1\gbmp t\gnl
  \gcl1\gcl1\gnl
  \gob1H\gob1H\gend
  =
  \gbeg3{11}
  \got3T\gnl
  \gvac1\gcl1\gnl
  \gvac1\gbmp t\gnl
  \gwcmh325\gnl
  \gcmu\gcl6\gnl
  \gmp+\gcl1\gnl
  \gmu\gnl
  \gcn2123\gnl
  \gvac1\gmp m\gnl
  \gvac1\gcl1\gnl
  \gvac1\gob1H\gob1H\gend
  =
  \gbeg39
  \got3T\gnl
  \gvac1\gcl1\gnl
  \gvac1\gbmp t\gnl
  \gwcmh314\gnl
  \gmp+\gcmu\gnl
  \gmp m\gmp m\gcl3\gnl
  \gdnot{\brd^m}\glmptb\grmptb\gnl
  \gmu\gnl
  \gob2H\gob1H\gend
  =
  \gbeg39
  \got3T\gnl
  \gvac1\gcl1\gnl
  \gvac1\gbmp t\gnl
  \gcn3134\gnl
  \gvac1\gcmu\gnl
  \gu1\gmp m\gcl3\gnl
  \gdnot{\brd^m}\glmptb\grmptb\gnl
  \gmu\gnl
  \gob2H\gob1H\gend
  =
  \gbeg38
  \got3T\gnl
  \gvac1\gcl1\gnl
  \gvac 1\gbmp t\gnl
  \gwcm 3\gnl
  \gmp m\gvac1\gcl2\gnl
  \gcl1\gnl
  \gob1H\gvac1\gob1H\gend
\end{equation*}

Since the braiding on $H$ has finite order by assumption, some even power of the antipode is a Hopf algebra endomorphism of $H$. Therefore that even power of the antipode is a monomorphism in $\B$. By the dual reasoning it is also an epimorphism, and therefore $S$ itself is an automorphism in $\B$.
\end{proof}

%
%
\section{Fitting's lemma}
\label{sec-fitting}
\begin{Prop}[Fitting's lemma]
  Let $H$ be a Hopf algebra, and $f$ a Hopf algebra endomorphism of $H$.
  
  Assume that $H$ satisfies both chain conditions, so that there is an $n\in\N$ such that $H\cong \Img(f^n)\ot \lcofix
    H{f^n} $ is a Radford biproduct.  
		
	If $f$ is normal, the action of $\Img(f^n)$ on $\lcofix H{f^n}$ is trivial, so that, as an algebra, $H$ is the tensor product of $\Img(f^n)$ and $\lcofix H{f^n}$ in $\B$. Similarly if $f$ is conormal, then the coalgebra $H$ is a tensor product of coalgebras in $\B$. In particular, if $f$ is binormal then $\Img(f^n)$ and $\lcofix H{f^n}$ are unbraided Hopf algebras in $\B$, and $H$ is isomorphic to their tensor product.
\end{Prop}
\begin{proof}
  We continue the proof of \cref{Radford}, assuming that $\Ker(f^2)=\Ker(f)$ and
  $\Coker(f^2)=\Coker(f)$ after replacing $f$ by a power of $f$. We
  now add the observation that
  normality of $f$ implies that $p=(f*(Sf*\id_H))p=(Sf*\id_H)p$. Therefore
  $f\comm p$, and dually $f\cocomm p$ if $f$ is conormal. This in turn implies that the
  Radford biproduct is just an ordinary tensor product algebra or
  tensor product coalgebra, as appropriate.
\end{proof}

\begin{Def}
  Let $H$ be a Hopf algebra. If $H\cong A\ot B$ for two Hopf algebras
  $A$ and $B$, then we say that $A$ is a tensor factor of $H$. We note
  that this implies that $A$ and $B$ are unbraided.

  We say that $H$ is tensor indecomposable if it does not have a
  nontrivial tensor factor.  An endomorphism $f$ of $H$ is nilpotent
  if there is $n\in\N$ such that $f^n=\eta\epsilon$.
\end{Def}

\begin{Cor}
  If $H$ is a tensor indecomposable Hopf algebra satisfying both chain conditions, then every binormal endomorphism of $H$ is nilpotent or an automorphism.
\end{Cor}

%
%
\section{Krull-Remak-Schmidt}
\label{sec:krull-remak-schmidt}

Of course, a Hopf algebra satisfying both chain conditions can be (inductively) decomposed as a tensor product of indecomposable Hopf subalgebras. We shall now show that the Hopf algebraic analog of the Krull-Remak-Schmidt theorem asserting the uniqueness of such a decomposition also holds.  A version of this for completely reducible semisimple Hopf algebras was established in \cite{Bur:CFPHA}.  In general, it cannot be hoped that this result has a categorical version.  In \cite{Mue:SMC} it was shown that a non-degenerate fusion category factorizes into a product of prime ones, but that this was generally not unique.  Therefore, such decompositions are rather specific to Hopf algebras.

\begin{Lem}
  Let $f,g$ be bicommuting, binormal endomorphisms of a tensor
  indecomposable Hopf algebra $H$.

  If $f$ and $g$ are nilpotent, then so is $f*g$.
\end{Lem}
\begin{proof}
  Otherwise $f*g$ is a normal automorphism, and after composing $f$
  and $g$ with its inverse, we can assume that $f*g=\id_H$. In particular $f$
  composition commutes with $g$. Then one can show by induction that
  \begin{math}
    \id_H=(f*g)^n
  \end{math}
  is a convolution product of terms of the form $f^kg^{n-k}$ for
  $0\leq k\leq n$ (in fact this is a binomial formula with binomial
  coefficients, but writing it is cumbersome because addition is
  replaced with convolution products). If $f^m=\eta\epsilon=g^m$, this
  implies $(f*g)^{2m}=\eta\epsilon$, since each term contains an
  $m$-th power of either $f$ or $g$.
\end{proof}

\begin{Rem}
  Let $H$ and $H_1\dots H_k$ be Hopf algebras in $\B$. Decomposing $H$ as a tensor product Hopf algebra
  \begin{equation*}
    H\cong H_1\ot H_2\ot\dots\ot H_k
  \end{equation*}
  amounts to specifying a system of injections $\iota_i\colon H_i\to H$ and projections $\pi_i\colon H\to H_i$, all of them Hopf algebra morphisms, which commute and cocommute pairwise, and satisfy $\pi_i\iota_i=\id_{H_i}$, $\pi_i\iota_j=\eta\epsilon$ if $i\neq j$, and $\iota_1\pi_1*\iota_2\pi_2*\dots*\iota_k\pi_k=\id_H$. The isomorphisms between $H$ and the tensor product are then given by
  \begin{equation*}
    H\xrightarrow{\Delta^{(k-1)}}H^{\ot k}\xrightarrow{\pi_1\ot\dots\pi_k}H_1\ot\dots\ot H_k
  \end{equation*}
  and
  \begin{equation*}
    H_1\ot\dots\ot H_k\xrightarrow{\iota_1\ot\dots\ot\iota_k}H^{\ot k}\xrightarrow{\nabla^{(k-1)}}H.
  \end{equation*}
  Note that the $H_i$ need to be pairwise unbraided for the tensor product Hopf algebra to make sense.
\end{Rem}

\begin{Thm}
  Let $H$ be a Hopf algebra in $\B$, and let
  \begin{align*}
    H&=H_1\ot H_2\ot \dots \ot H_k\\
    &=G_1\ot G_2\ot\dots \ot G_\ell
  \end{align*}
  be two tensor decompositions of $H$ in tensor indecomposable factors.

  Then $k=\ell$, and $H_i\cong G_i$ after a suitable permutation of
  the indices.

  Moreover, if
  \begin{equation*} \xymatrix{H_i\ar@<.5ex>[r]^{\iota_i}&H\ar@<.5ex>[l]^{\pi_i}\ar@<.5ex>[r]^{p_j}&G_j\ar@<.5ex>[l]^{q_j}}
  \end{equation*}
  denote the systems of injections and projections going with the
  decompositions into tensor factors, then the factors can be so numbered that for any $1\leq m\leq k$
  \begin{equation}
    H\xrightarrow{\Delta^{(k-1)}}H^{\ot k}\xrightarrow{\pi_1\ot\dots\ot\pi_m\ot p_{m+1}\ot\dots\ot p_k}H_1\ot\dots H_m\ot G_{m+1}\ot\dots\ot G_k\label{eq:2}
    \end{equation}
    and
    \begin{equation}
    H_1\ot\dots H_m\ot G_{m+1}\ot\dots\ot G_k\xrightarrow{\iota_1\ot\dots\ot\iota_m\ot q_{m+1}\ot\dots\ot q_k}H^{\ot k}\xrightarrow{\nabla^{(k-1)}}H\label{eq:1}
  \end{equation}
  are isomorphisms.
\end{Thm}
\begin{proof}
  There is nothing to show if one of the decompositions consists of
  only one factor. Otherwise we consider
  \begin{equation*}
    \id_{H_1}=\pi_1\iota_1=\pi_1(q_1p_1*\dots*q_\ell p_\ell)\iota_1
      =\pi_1q_1p_1\iota_1*\dots*\pi_1q_\ell p_\ell\iota_1.
  \end{equation*}
  Since $H_1$ is indecomposable, and the terms in the last convolution
  product are bicommuting binormal endomorphisms, we know
  that one of $\pi_1q_jp_j\iota_1$ is an isomorphism. Without loss of
  generality we assume this happens for $j=1$, and that
  $\pi_1q_1p_1\iota_1=\id_{H_1}$. It follows that $\pi_1q_1$ and
  $p_1\iota_1$ are mutually inverse isomorphisms between $H_1$ and
  $G_1$. Now put $f=q_2p_2*\dots*q_\ell p_\ell$ and
  $t=\iota_1\pi_1q_1p_1*f$. Since $p_1t=p_1\iota_1\pi_1q_1p_1=p_1$, we
  have $\rcofix Ht\subset\rcofix H{p_1t}=\rcofix H{p_1}$. Thus, for
  $j\colon\rcofix Ht\to H$ the inclusion, we find
  \begin{equation*}
    \Delta j=(H\ot q_1p_1*\dots*q_\ell p_\ell)\Delta j
    =(H\ot f)\Delta j
    =(H\ot t)\Delta j
    =(H\ot\eta)j,
  \end{equation*}
  and therefore $\rcofix Ht$ is trivial. We conclude that $t$ is an
  automorphism of $H$. 

  Write $\tilde\pi\colon H\to H_2\ot\dots\ot H_k=:\tilde H$ and
  $\tilde\iota\colon \tilde H\to H$ for the natural
  projection and injection morphisms, and similarly for $\tilde
  p\colon H\to \tilde G$, $\tilde q\colon \tilde G\to H$. Since $tq_1=\iota_1\pi_1q_1$, we
  have $\tilde\pi tq_1=\eta\epsilon$, and thus $\tilde\pi t=\tilde\pi
  t\tilde q\tilde p$ and $\tilde\pi t\tilde q\tilde
  pt\inv\tilde\iota=\tilde\pi\tilde\iota=\id_{\tilde H}$. It follows
  that $\tilde\pi t\tilde q$ and $\tilde p t\inv\tilde\iota$ are
  mutually inverse isomorphism between $\tilde G$ and $\tilde H$.

  Thus, by an inductive argument we have $k=\ell$, and we can rearrange the indices to get $H_i\cong G_i$ for all $i$.

  Note further that the automorphism $t$ above is the composition of the isomorphism
  \begin{equation*}
    H\xrightarrow{\Delta^{(k-1)}}H^{\ot k}\xrightarrow{p_1\ot\dots\ot p_k}G_1\ot\dots\ot G_k\xrightarrow{\pi_1q_1\ot G_2\ot\dots\ot G_k}H_1\ot G_2\ot\dots\ot G_k
  \end{equation*}
  with the morphism
  \begin{equation*}
    H_1\ot G_2\ot\dots\ot G_k\xrightarrow{\iota_1\ot q_2\ot\dots\ot q_k}H^{\ot k}\xrightarrow{\nabla^{(k-1)}}H,
  \end{equation*}
  whence the latter is an isomorphism. Again by an inductive argument, we get that \eqref{eq:1} is an isomorphism; the reasoning for \eqref{eq:2} is similar.
\end{proof}

%
%
\section{Endomorphisms of tensor products}
\label{sec:endo-tens-prod}

Let $H$ and $K$ be two Hopf algebras in $\B$, unbraided so that one
can form the tensor product bialgebra $H\ot K$. Let $A$ be an algebra in $\B$.
It is well-known that there is a bijection
\[\Alg(H\ot K,A)\cong\{(a,b)\in \Alg(H,A)\times\Alg(K,A)|a\comm b\}.\]
In fact, a pair $(a,b)$ of commutation commuting algebra morphisms induces $f=\nabla_A(a\ot b)$, and
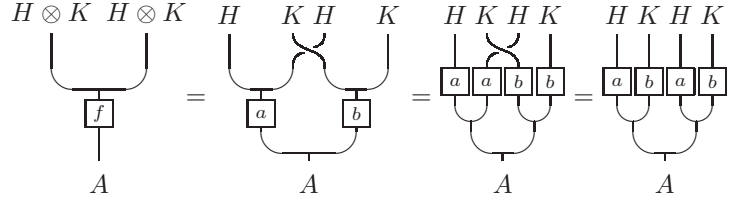
\begin{equation*}
  \gbeg56
  \got2{H\ot K}\gvac1\got2{H\ot K}\gnl
  \gcn2122\gvac1\gcn2122\gnl
  \gwmuh528\gnl
  \gvac2\gbmp{f}\gnl
  \gvac2\gcl1\gnl
  \gvac2\gob1A\gend
= 
  \gbeg66
  \got1H\gvac1\got1K\got1H\gvac1\got1K\gnl
  \gcl1\gvac1\gbr\gvac1\gcl1\gnl
  \gwmu3\gwmu3\gnl
  \gvac1\gbmp a\gvac2\gbmp b\gnl
  \gvac1\gwmu4\gnl
  \gvac1\gob4A\gend
=
  \gbeg46
  \got1H\got1K\got1H\got1K\gnl
  \gcl1\gbr\gcl1\gnl
  \gbmp a\gbmp a\gbmp b\gbmp b\gnl
  \gmu\gmu\gnl
  \gwmuh426\gnl
  \gob4A\gend
=
  \gbeg46
  \got1H\got1K\got1H\got1K\gnl
  \gcl1\gcl1\gcl1\gcl1\gnl
  \gbmp a\gbmp b\gbmp a\gbmp b\gnl
  \gmu\gmu\gnl
  \gwmuh426\gnl
  \gob4A\gend
\end{equation*}
shows that $f$ is multiplicative. Conversely, given $f\colon H\ot K\to
A$ define $a=f(H\ot \eta)$ and $b=f(\eta\ot K)$. Then, with $T:=H\ot K$:
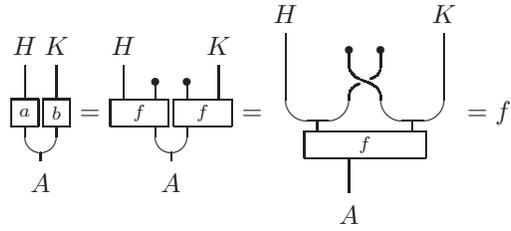
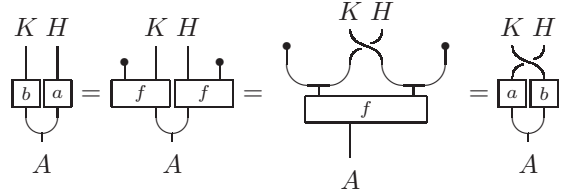
\begin{equation*}
  \gbeg25
  \got1H\got1K\gnl
  \gcl1\gcl1\gnl
  \gbmp a\gbmp b\gnl
  \gmu\gnl
  \gob2A\gend
=
  \gbeg45
  \got1H\gvac2\got1K\gnl
  \gcl1\gu1\gu1\gcl1\gnl
  \gdnot f\glmpt\grmptb\gdnot f\glmptb\grmpt\gnl
  \gvac1\gmu\gnl
  \gob4A\gend
=
  \gbeg67
  \got1H\gvac4\got1K\gnl
  \gcl2\gvac1\gu1\gu1\gvac1\gcl2\gnl
  \gvac2\gbr\gnl
  \gwmu3\gwmu3\gnl
  \gvac1\glmpt\gdnot f\gcmpb\gcmp\grmpt\gnl
  \gvac2\gcl1\gnl
  \gvac2\gob1A\gend
=f
\end{equation*}
and
\begin{equation*}
  \gbeg25
  \got1K\got1H\gnl
  \gcl1\gcl1\gnl
  \gbmp b\gbmp a\gnl
  \gmu\gnl
  \gob2A\gend
=
  \gbeg45
  \gvac1\got1K\got1H\gnl
  \gu1\gcl1\gcl1\gu1\gnl
  \gdnot f\glmpt\grmptb\gdnot f\glmptb\grmpt\gnl
  \gvac1\gmu\gnl
  \gob4A\gend
=
  \gbeg66
  \gvac2\got1K\got1H\gnl
  \gu1\gvac1\gbr\gvac1\gu1\gnl
  \gwmu3\gwmu3\gnl
  \gvac1\glmpt\gdnot f\gcmpb\gcmp\grmpt\gnl
  \gvac2\gcl1\gnl
  \gvac2\gob1A\gend
=
  \gbeg25
  \got1K\got1H\gnl
  \gbr\gnl
  \gbmp a\gbmp b\gnl
  \gmu\gnl
  \gob2A\gend
\end{equation*}

Assume that $A$ is a bialgebra in $\B$, and $a,b,f$ are as above. Then
$f$ is a bialgebra homomorphism if and only if $a$ and $b$ are.

Dually, for a coalgebra $C$ in $\B$, a bijection 
\[\left\{(a,b)\in\Coalg(C,H)\times\Coalg(C,H)\left|
 a\cocomm b
\right.\right\}\longrightarrow\Coalg(C,H\ot K)
\]
is given by $(a,b)\mapsto (a\ot b)\Delta$, and it induces bijections
on the subsets containing (pairs of) bialgebra maps.

Putting the above together, one obtains a bijection between
$\Hopfend(H\ot K)$ and 
\begin{equation*}
  \left\{(a,b,c,d)\left|\begin{array}{l}a\in\Hopfend(H),\ 
                                        b\in\Hopfmor(K,H),\\
                                        c\in\Hopfmor(H,K),\ 
                                        d\in\Hopfend(K),\\
a\comm b,c\comm d,
a\cocomm c,b\cocomm d
\end{array}
\right.\right\}
\end{equation*}
with the endomorphism of $H\ot K$ corresponding to a quadruple of Hopf
algebra map ``components'' given by
\begin{equation*}
  \begin{pmatrix}a&b\\c&d\end{pmatrix}
  :=
  \gbeg46
  \got2H\got2K\gnl
  \gcmu\gcmu\gnl
  \gcl1\gbr\gcl1\gnl  
  \gbmp a\gbmp b\gbmp c\gbmp d\gnl
  \gmu\gmu\gnl
  \gob2H\gob2K\gend.
\end{equation*}

Consider a second endomorphism $g$ of $H\ot K$ dissected analogously
into a matrix
\begin{math}
  \begin{pmatrix}
    a'&b'\\c'&d''
  \end{pmatrix}
\end{math}
of Hopf algebra endomorphisms. Then it is straightforward to check
that $gf$ corresponds to
\begin{math}
  \begin{pmatrix}
    a'a*b'c&a'b*b'd\\
    c'a*d'c&c'b*d'd
  \end{pmatrix}.
\end{math}

\begin{Prop}
  Let $H$ and $K$ be as above, and $f\in\Hopfend(H\ot K)$ described by a matrix $
  \begin{pmatrix}
    a&b\\c&d
  \end{pmatrix}$. Assume that the antipodes of $H$ and $K$ are automorphisms in $\B$.
  
  Then $f$ is normal if and only if $a$ and $c$ are normal, $b\comm\id_H$, and $d\comm\id_K$; a similar characterization holds for conormal endomorphisms.
\end{Prop}
\begin{proof}
  We fix projections and injections for the tensor product $P:=H\ot K$:
  \begin{equation*} \xymatrix{H\ar@<.5ex>[r]^{\iota_H}&P\ar@<.5ex>[l]^{\pi_H}\ar@<.5ex>[r]^{\pi_K}&K\ar@<.5ex>[l]^{\iota_K}}
  \end{equation*}
  First assume that $f$ is normal. Since $f\comm (fS*\id_P)$, $a=\pi_Hf\iota_H$ commutes with $\pi_H(fS*\id_P)\iota_H=aS*\id_H$. Similarly $c$ is normal. Using \eqref{eq:3} we have
  \begin{equation*}
    \gbeg25
    \got1H\got1K\gnl
    \gbr\gnl
    \gbmp b\gcl1\gnl
    \gmu\gnl
    \gob2H\gend
    =
    \gbeg5{11}
    \got3H\got1K\gnl
    \gvac1\gcl1\gvac1\gcl1\gnl
    \gvac1\gbmp{\iota_H}\gvac1\gbmp{\iota_K}\gnl
    \gwcm3\gcl1\gnl
    \gcl1\gvac1\gbr\gnl
    \glmpt\gnot{\operatorname{ad}}\gcmp\grmptb\gcn2213\gnl
    \gvac2\gbmp f\gnl
    \gvac2\gwmu3\gnl
    \gvac3\gbmp{\pi_H}\gnl
    \gvac3\gcl1\gnl
    \gvac3\gob1H\gend
    =
    \gbeg38
    \got1H\gvac1\got1K\gnl
    \gcl1\gvac1\gcl1\gnl
    \gbmp{\iota_H}\gvac1\gbmp{\iota_K}\gnl
    \gcl1\gvac1\gbmp f\gnl
    \gwmu3\gnl
    \gvac1\gbmp{\pi_H}\gnl
    \gvac1\gcl1\gnl
    \gob3H\gend
    =
    \gbeg25
    \got1H\got1K\gnl
    \gcl2\gcl1\gnl
    \gvac1\gbmp b\gnl
    \gmu\gnl
    \gob2H\gend
  \end{equation*}
  so that $b\comm\id_H$. Similarly $d\comm\id_K$.

  Now suppose that the stated normality and commutation conditions on $a,b,c,d$ hold. Writing $\hat a=\iota_Ha\pi_H,\hat b=\iota_Hb\pi_K$ etc.\ we can write $f=\hat a*\hat b*\hat c*\hat d$ as a convolution product of four commuting and cocommuting endomorphisms of $P$. We are claiming that this product commutes with
  \begin{equation*}
    fS*\id_P=\hat aS*\hat bS*\hat cS*\hat dS*\iota_K\pi_K*\iota_H\pi_H
      =\hat bS*\hat cS*\hat dS*\iota_K\pi_K*\hat aS*\iota_H\pi_H.
  \end{equation*}
  (the last equality using that $\hat a$ bicommutes with $\hat b,\hat c,\hat d,$ and $\iota_K$.)
  Now $\hat a$ commutes with $\hat aS*\iota_H\pi_H$ since $a$ is normal, with $\hat bS$ since $a\comm b$, and with $\iota_K\pi_K$ and $\hat cS$ since $\iota_H\comm\iota_K$. The next factor $\hat b$ commutes with $\hat aS$, $\hat bS$, and $\iota_H\pi_H$ since $b\comm\id_H$, and it commutes with $\hat cS,\hat dS,$ and $\iota_K\pi_K$, since $\iota_K\comm\iota_H$. Similar arguments deal with the convolution factors $\hat c$ and $\hat d$.
\end{proof}

\begin{Rem}
  Similarly, an endomorphism $f$ of a tensor product of several pairwise unbraided Hopf algebras $H_1,\dots,H_k$ can be described by a matrix $(v_{ij})$ of Hopf algebra homomorphisms between the factors. By inductive arguments one can show that $f$ is normal iff all the diagonal terms are normal, and the off-diagonal terms commute with the identities on their targets.

  An interesting case arises when there are no nontrivial homomorphisms $H_i\to H_j$ commuting with $\id_{H_j}$. In this case any normal endomorphism preserves the decomposition into tensor factors. One can deduce from this that the Krull-Remak-Schmidt decomposition is unique in a stronger sense than up to permutation and isomorphism; in the original case of decompositions of groups the uniqueness of the subgroups in a direct decomposition into directly indecomposable factors follows as stated in Remak's thesis \cite{Rem:ZEGDUF}.
\end{Rem}

%
%
\section{Automorphisms of tensor products}
\label{sec:autom-tens-prod}
We consider now the automorphisms of tensor products of Hopf algebras.  These are the natural extensions of the corresponding results in group theory \cite{BidCurMcC:ADPFG,Bid:ADPFG2}.

Throughout this section we let $H$ and $K$ be unbraided Hopf algebras, so that we can form the tensor product $H\ot K$, and we assume that the antipodes of $H$ and $K$ are automorphisms in $\B$.

Identify endomorphisms of $H\ot K$ with matrices of Hopf algebra homomorphisms as in \cref{sec:endo-tens-prod}.
Let
\begin{math}
  \begin{pmatrix}
    a&b\\c&d
  \end{pmatrix}\in\Hopfend(H\ot K).
\end{math}
If $a$ is an automorphism, then by \cref{compositescommute} of \cref{usefulsimplecommutationfacts} the condition $a\comm b$ implies $\id_H\comm b$ (and $x\comm b$ for any $x\colon X\to H$). Similarly $\id_K\cocomm c$, and, if $d$ is also an automorphism, $b\cocomm\id_H$ and $c\comm \id_K$.

Define \[\mathcal{A} = \begin{pmatrix}
      \Hopfaut(H)&\Zenthom(K,H)\\
      \Zenthom(H,K)&\Hopfaut(K)
    \end{pmatrix},\]
where $\Zenthom(K,H):=\{b\in\Hopfmor(K,H)|b\comm\id_H\text{ and
}b\cocomm\id_H\}$.  This is easily seen to be an abelian group under convolution product.  Indeed, the image of any such morphism determines an abelian Hopf sub-algebra of $H$.  Note that $b\comm \id_H$ $\iff$ $b\comm \alpha$ for some/all $\alpha\in\Hopfaut(H)$, and similarly $b\cocomm \id$ $\iff$ $b\cocomm \alpha$ for some/all $\alpha\in\Hopfaut(H)$.

Consider an automorphism $f$ of $H\ot K$, and its decomposition as a matrix 
$\begin{pmatrix}
  a&b\\c&d
\end{pmatrix}$
of Hopf algebra homomorphisms as in \cref{sec:endo-tens-prod}. Let
$f\inv$ correspond in the same way to a matrix
\begin{math}
  \begin{pmatrix}
    a'&b'\\c'&d'
  \end{pmatrix}.
\end{math}
Then we have $\id_K=(\epsilon \ot K)f\inv f(\eta\ot K)=c'b*d'd$. Since
$c'\comm d'$ and $b\cocomm d$, we have that $c'b\comm
d'd=(c'bS)*\id_K$ and $c'b\cocomm (c'bS)*\id_K$. In other words, $c'b$
is a binormal endomorphism of $K$. In the same way $bc'$ is
a binormal endomorphism of $H$. If we assume both chain conditions on $H$ and $K$, then for sufficiently large $n$, $b$
and $c'$ induce mutually inverse isomorphisms between the images of
$(c'b)^n$ and $(bc')^n$. Thus, using Fittings Lemma, the image of
$(c'b)^n$ is a common tensor factor of $H$ and $K$. 

This gives part of the following result.

\begin{Thm}
  Suppose that $H$ and $K$ satisfy both chain conditions. Then $\mathcal{A}\subseteq \Hopfaut(H\ot K)$ if and only if $H$ and $K$ have no non-trivial common abelian direct tensor factors.  On the other hand, $\Hopfaut(H\ot K)=\mathcal{A}$ if and only if $H$ and $K$ have no non-trivial common direct tensor factors.
\end{Thm}
\begin{proof}
  If $H$ and $K$ have a common non-trivial direct tensor factor, then permutations of this factor in $H\ot K$ are automorphisms of $H\ot K$ not contained in $\mathcal{A}$.
	
  By the preceeding remarks, to show $\Hopfaut(H\ot K)\subseteq \mathcal{A}$ it remains to prove that the common tensor factor in $H\ot K$ that we found is necessarily nontrivial if $d$ is not an automorphism. A similar argument will apply to show that $a$ is an automorphism, and the commutation and cocommutation conditions for the components of an endomorphism will be equivalent to the off-diagonal terms (co)commuting with the identity instead of the automorphisms on the diagonal.

  Thus suppose that $d$ is not an automorphism. Then we can assume without loss of generality that the right $d$-coinvariant subobject  $D$ of $H$ is nontrivial. If $\iota\colon D\to H$ is the inclusion, then $c'b\iota=\nabla(c'b\ot\eta)\iota=\nabla(c'b\ot d'd)\Delta\iota=(c'b*d'd)\iota=\iota$, hence $(c'b)^n\iota=\iota$ for all $n$, and the image of $(c'b)^n$ is nontrivial as desired.
	
The desired equality in the second part will then hold once we have proven the first equivalence.

	To this end we first consider the forward direction by contrapositive.  Suppose that $H$ and $K$ have a common abelian direct tensor factor $L$, and write $H=H'\ot L$ and $K=K'\ot L$.  Since $L$ is abelian its antipode $S_L$ is a Hopf endomorphism of $L$.  Taking $a=\id_H,d=\id_K$, $b=\eta_{K'}\varepsilon_{K'}\ot S_L$ and $c=\eta_{H'}\varepsilon_{H'}\ot S_L$ we find that $\psi=\begin{pmatrix} a&b\\c&d \end{pmatrix}\in\mathcal{A}$.  However, $L$ is a sub-object of the right $\psi$-coinvariant subobject, whence $\psi\not\in\Hopfaut(H\ot K)$.

  For the remaining direction, assume that $f=
  \begin{pmatrix}
    a&b\\c&d
  \end{pmatrix}$ belongs to $\mathcal{A}$; in particular $f$ is a Hopf algebra endomorphism of $H\ot K$. After multiplying with the obvious automorphism $
  \begin{pmatrix}
    a\inv&\eta\epsilon\\\eta\epsilon&d\inv
  \end{pmatrix}$ of $H\ot K$ we may assume that $a=\id_H$ and $d=\id_K$. Now consider $g=
  \begin{pmatrix}
    \id_H&bS\\cS&\id_K
  \end{pmatrix}$, another Hopf endomorphism of $H\ot K$. One computes
  $gf=
  \begin{pmatrix}
    \id*bcS&b*bS\\cS*c&cbS*\id
  \end{pmatrix}
  =
  \begin{pmatrix}
    \id*bcS&\eta\epsilon\\\eta\epsilon&cbS*\id
  \end{pmatrix}.$
  By the chain conditions on $H$ and $K$, for $n$ sufficiently large $b$ and $c$ induce mutually inverse isomorphisms between the images of $(bc)^n$ and $(cb)^n$. Fitting's lemma implies that these isomorphic images are an abelian common tensor factor of $H$ and $K$. It can only be trivial if $bc$ and $cb$ are nilpotent, in which case $\id*bcS$ and $cbS*\id$ are automorphisms. In the latter case, $f$ was an automorphism.
\end{proof}

These results have obvious extensions to more than two factors by induction, which we leave to the reader.  The results, however, do not cover the case of a repeated tensor factor.  For a given Hopf algebra $H$ in $\B$ we can form the unbraided iterated tensor product $H^{\ot n}=H\ot\cdots\ot H$ for $n\in\N$ precisely when $H$ is in a (sub)category where the braiding is a symmetry.

\begin{Thm}
	Let $H$ be a tensor indecomposable non-abelian Hopf algebra satisfying both chain conditions in $\B$, and suppose the braiding of $\B$ is a symmetry.  Fix $n\in\N$, and let $\mathcal{A}_n$ denote those $(\alpha_{ij})\in\Hopfend(H^{\ot n})$ such that $\alpha_{ii}\in\Hopfaut(H)$ and $\alpha_{ij}\in\Hopfend(H)$ for all $i$ and $j\neq i$.  Then
	\[\Hopfaut(H^{\ot n}) \cong \mathcal{A}_n\rtimes S_n.\]
\end{Thm}
\begin{proof}
	By assumptions on $H$, $\mathcal{A}_n\subseteq\Hopfaut(H^{\ot n})$.  The group $S_n$ acts on $H^{\ot n}$ by permuting factors, and so acts on $\Hopfaut(H^{\ot n})$ by permuting columns.  Conjugating by this action sends $\mathcal{A}_n$ to itself.  We need only show that every automorphism is a column permutation of an element of $\mathcal{A}_n$.
	
	So let $(\alpha_{ij})\in\Hopfaut(H^{\ot n})$, with inverse $(\alpha_{ij}')$.  Then for all $i$ we have $\alpha_{i1}\alpha_{1i}'*\cdots*\alpha_{in}\alpha_{ni}'=\id$.  Since the $\alpha_{ik}\alpha_{ki}'$ are all binormal endomorphisms the notation is unambiguous, and the terms of the convolution product can be arbitrarily reordered.  Moreover, since $H$ is indecomposable we may conclude that one of the $\alpha_{ik}\alpha_{ki}'$ is an automorphism.  In particular for all $i$ there is a $k$ such that $\alpha_{ik}$ is an epimorphism and $\alpha_{ki}'$ is a monomorphism.  By the chain conditions it follows that $\alpha_{ik}$ and $\alpha_{ki}'$ are both automorphisms.  Since $H$ is non-abelian there is at most one such $k$ for any given $i$.  This completes the proof.
\end{proof}

%
%
\section{Application to doubles of groups}
\label{sec:appl-doubl-groups}
For this section we work in the category of vector spaces over a field $\kk$.   Throughout this section $G,H,K,C$ will all be finite groups.  For any group $G$ let $\widehat{G}$ be the group of group-like elements of $\du{G}$, the dual of the group algebra $\kk G$.   Note that $\widehat{G}$ is precisely the $\kk$-linear characters of $G$.   We also define $\Gamma_G=\widehat{G}\times G$. We denote the conjugation action of $G$ on $\du{G}$ and $\kk G$ both by $\rightharpoonup$; e.g. $g\rightharpoonup x = gxg^{-1}$ for all $g,x\in G$. We will be concerned with $\D(G)$, the Drinfeld double of a finite group $G$.  As a coalgebra $\D(G)=\rcofix{\du{G}}{}\ot \kk G$, and the algebra structure is given by having $G$ act on $\rcofix{\du{G}}{}$ by the conjugation action.  We note that $\Gamma_G$ is the group of group-like elements of $\D(G)$.  See \cite{DijPasRoc:QHAGCOM,Mon:HAAR} for further details on the construction and properties of this Hopf algebra.  

In \cite{K14} the first author gave a complete description of $\Hopfaut(\D(G))$ whenever $G$ has no non-trivial abelian direct factors.  Such a group is said to be purely non-abelian.  When $G$ is abelian we have $\D(G)=\du{G}\ot\kk G$, an abelian Hopf algebra, and the determination of $\Hopfaut(\D(G))$ is then straightforward.  Indeed, under mild assumptions on $\kk$ we have $\D(G) \cong \kk(G\times G)$. Subsequently in this case $\Hopfaut(\D(G))$ can be computed by classical methods in group theory \cite{Sho:AAG}. We note that the structure of such an automorphism group has been of more recent interest \cite{BidCur:AFAG,HilRhe:AFAG}.  It is the goal of this section to complete the description of $\Hopfaut(\D(G))$ when $G$ has an abelian direct factor.

So suppose that $G=C\times H$ with $C$ abelian.  Then $\D(G)\cong\D(C)\otimes\D(H)$.  Since $\D(C)$ is an abelian Hopf algebra the results of the previous section can be applied whenever $\D(H)$ has no abelian direct tensor factors.  We will proceed to show this happens precisely when $H$ is purely non-abelian.

We have the following description of $\Hopfmor(\D(G),\D(K))$.

\begin{Thm}[\cite{K14}]
The elements of $\Hopfmor(\D(G),\D(K))$ are in bijective correspondence with matrices 
$\begin{pmatrix}
    u&r\\p&v
  \end{pmatrix}$
where $u\colon \du{G}\to\du{K}$ is a morphism of unitary coalgebras; $p\colon \du{G}\to\kk K$ is a morphism of Hopf algebras; and $r\colon G\to\widehat{K}$ and $v\colon G\to K$ are group homomorphisms.  These are all subject to the following compatibility relations, for all $a,b\in\du{G}$ and $g\in G$:
\begin{enumerate}
	\item $u(g\rightharpoonup a)=v(g)\rightharpoonup u(a)$, from which it follows that $u^*v$ is normal;
	\item $u\cocomm p$;
	\item $u(ab)=u(a_{(1)})(p(a_{(2)})\rightharpoonup u(b))$;
	\item $p(g\rightharpoonup a)=v(g)\rightharpoonup p(a)$.
\end{enumerate}
The morphism is defined by \[a\# g\mapsto u(a_{(1)})r(g)\# p(a_{(2)})v(g).\]
Composition of such morphisms is given by matrix multiplication, as in \cref{sec:endo-tens-prod}.
\end{Thm}
The morphism $p$ is uniquely determined by an isomorphism $\du{A}\cong\kk B$, where $A$ is an abelian normal subgroup of $G$ and $B$ is an abelian subgroup of $K$.  In particular we must have $\du{A}\cong \kk \widehat{A}$.  For the remainder of this section any use of $A,B$ refers to these subgroups.  We note that the last relation says $p\comm v$ if and only if $A\leq Z(G)$, or equivalently $p$ is cocentral: $p\cocomm \id$.

By convention we implicitly identify any element of $\Hopfmor(\D(G),\D(K))$ or $\Hopfmor(\tprod{G},\tprod{K})$ with its quadruple of components $(u,r,p,v)$, or equivalently as a matrix $\begin{pmatrix} u&r\\p&v \end{pmatrix}$.

The following is then immediate. 

\begin{Lem}
	A morphism $\psi\in\Hopfmor(\D(G),\D(K))$ is canonically an element of $\Hopfmor(\tprod{G},\tprod{K})$ precisely when $p\comm v$ and $u$ is a morphism of Hopf algebras.
	
	On the other hand, $\phi\in\Hopfmor(\tprod{G},\tprod{K})$ is canonically an element of $\Hopfmor(D(G),\D(K))$ precisely when $u^*\circ v$ is normal and $A\leq Z(G)$.
\end{Lem}

In the first case we call such a morphism untwistable, and in the second we call it twistable.  Clearly any untwistable morphism is also twistable, and vice versa.  The distinction is simply in the algebra structures we start with.

Now since $\tprod{G}$ and $\tprod{K}$ are canonically self-dual any morphism $\psi\in\Hopfmor(\tprod{G},\tprod{K})$ yields a dual morphism $\psi^*\in\Hopfmor(\tprod{K},\tprod{G})$ with components $(v^*,r^*,p^*,u^*)$.  The following is then clear.

\begin{Cor}
	Both $\psi\in\Hopfmor(\tprod{G},\tprod{K})$ and $\psi^*$ are twistable if and only if the following all hold
	\begin{enumerate}
		\item $u^* v$ is normal;
		\item $v u^*$ is normal;
		\item $A\leq Z(G)$;
		\item $B\leq Z(K)$.
	\end{enumerate}
In this case we may canonically view $\psi\in\Hopfmor(\D(G),\D(K))$ and $\psi^*\in\Hopfmor(\D(K),\D(G))$.
\end{Cor}

In \cite{K14} a morphism $\psi=(u,r,p,v)\in\Hopfmor(\D(G),\D(K))$ was said to be flippable if also $(v^*,r^*,p^*,u^*)\in\Hopfmor(\D(K),\D(G))$.  This is equivalent to saying that $\psi$ is untwistable and the corresponding dual $\psi^*$ is twistable.  In particular the Corollary gives a complete description of the flippable elements of $\Hopfmor(\D(G),\D(K))$, and shows that 'flipping' an element of $\Hopfmor(\D(G),\D(K))$ can naturally be described as dualizing the morphism.

\begin{Cor}
	For any group $G$, $\Hopfaut(\D(G))$ is canonically a subgroup of $\Hopfaut(\tprod{G})$ which is closed under dualization.
\end{Cor}
\begin{proof}
	Follows from the preceeding corollary, \cref{sec:autom-tens-prod}, and the properties of $\Hopfaut(\D(G))$ established in \cite{K14}.
\end{proof}
	
We now show that the act of untwisting a morphism is fairly well-behaved whenever the image is commutative.
	
	\begin{Prop}\label{prop:algebra-change-preserves}
		Let $\psi\in\Hopfmor(\D(G),\D(K))$ be untwistable.  For convenience, let $\psi'=\psi\in\Hopfmor(\tprod{G},\tprod{H})$.  Then the following all hold.
		\begin{enumerate}
		  \item If $G=H$, then $\psi$ is conormal if and only if $\psi'$ is conormal.\label{prop:conormal-preserved}
			\item If $\psi$ has a commutative image then $\psi'$ has commutative image.
			\item If $\psi$ has commutative image and $G=H$, then $\psi'$ is normal if and only if $v$ is normal and $B\leq Z(G)$.
			\item If $\psi$ has commutative image and $G=H$, then $\psi$ is normal if and only if $\psi'$ is normal and $G$ acts trivially on $\operatorname{Img}(u)$.
			\item If $\psi$ has commutative image and $G=H$, then $\psi$ is conormal.
		\end{enumerate}
	\end{Prop}
	\begin{proof}
		We first prove \cref{prop:conormal-preserved}, as it is the only one that does not suppose that $\psi$ has commutative image.  To this end we compute
		\begin{align}
			a_{(3)}\# g \cdot S(a_{(1)}\# g) \otimes a_{(2)}\# g &= a_{(3)}\#g\cdot (g^{-1}\rightharpoonup S(a_{(1)})\#g^{-1})\otimes a_{(2)}\# g\nonumber\\
			&= a_{(3)}S(a_{(1)})\#1 \otimes a_{(2)}\# g;\label{eq:double-conormal-base}\\
			a_{(3)}\otimes g \cdot S(a_{(1)}\otimes g) \otimes a_{(2)}\otimes g
			&= a_{(3)}S(a_{(1)})\otimes 1\otimes a_{(2)}\otimes g\label{eq:tensor-conormal-base}.
		\end{align}
	The claim then follows.		
	
	For the remainder of the proof, suppose that $\psi$ is untwistable and has commutative image.  By checking the commutativity condition we can easily determine the following facts: $p\comm v$; $v$ has abelian image, or equivalently $v\comm v$; $u(a(g\rightharpoonup b))=u((h\rightharpoonup a)b)=u(ab)$ for all $g,h\in G$ and $a,b\in\du{G}$, which implies $u\cocomm \id$.  In particular, $u(g\rightharpoonup a)=u(a)$ and $p(g\rightharpoonup a)=p(a)$ for all such $a,g$.
	
	The first part is then immediate, as we have $\psi(a\#g \cdot b\# h) = \psi'(a\otimes g \cdot b\otimes h)$.  Another way of saying this is that when $\psi$ has commutative image we may compute products in either the double or the tensor product without affecting the result.  Furthermore $\psi((g\rightharpoonup a)\# g)=\psi(a\# g)$ for all appropriate $a,g$, and so \[\psi(S(a\# g)) = \psi(g^{-1}\rightharpoonup S(a)\# g^{-1})=\psi(S(a)\# g^{-1})=\psi'(S(a\ot g)).\]  Thus we may perform all computations with $\psi$ in either $\D(G)$ or $\tprod{G}$ as we desire.
	
	The last part of the result follows from \cref{eq:double-conormal-base,eq:tensor-conormal-base} and $u\cocomm \id$.  We need only prove the parts concerning normality of $\psi,\psi'$.
	
	To determine when $\psi'$ is normal, we first note that by commutativity we have
	\[ \psi'(a_{(2)}\otimes g\cdot b\otimes h \cdot S(a_{(1)})\otimes g^{-1}) = a(1)\psi'(b\otimes h).\]
	On the other hand,
	\[
		a_{(2)}\otimes g \cdot \psi'(b\otimes h) \cdot S(a_{(1)})\otimes g^{-1} = a(1) r(h)u(b_{(1)})\otimes g p(b_{(1)}) v(h) g^{-1}.
	\]
$\psi'$ is normal precisely when these two expressions are the same, and we easily find this is equivalent to $B\leq Z(H)$ and $v$ normal.
	
	Finally, we determine when $\psi$ is normal.  By previous remarks, we have	
	
	\begin{align}
		\psi(a_{(2)}\#g \cdot b\# h \cdot S(a_{(1)}\#g)) &= a(1) \psi(b\# ghg^{-1})\nonumber\\
		&= a(1) r(h)u(b_{(1)})\# p(b_{(2)})v(ghg^{-1})\label{eq:psi-normal-inside}.
	\end{align}
 On the other hand, we have
	\begin{align*}
		&a_{(2)}\#g \cdot \psi(b\# h) \cdot \left(g^{-1}\rightharpoonup Sa_{(1)}\# g^{-1}\right)\nonumber\\& \qquad= r(h) a_{(2)} \left(g\rightharpoonup u(b_{(1)})\right) (g p(b_{(2)})v(h)g^{-1}\rightharpoonup S(a_{(1)}))\# g p(b_{(3)})v(h)g^{-1}.
	\end{align*}
	Applying $\epsilon\#\id$ to both expressions we get \[a(1)p(b)v(ghg^{-1})\] for the first and  \[a(1) g p(b) v(h) g^{-1}\] for the second.  These are equal for all $a,b,g,h$ if and only if $B\leq Z(H)$ and $v$ is normal; equivalently, $\psi'$ is normal.  Note that if $v$ is normal and has abelian image, then its image is in fact central.  Therefore $g p(b)v(h)g^{-1}\rightharpoonup S(a) = S(a)$ precisely when $\psi'$ is normal.  Subsequently the previous equation simplifies to 
	\[ a(1) r(h)(g\rightharpoonup u(b_{(1)}))\# p(b_{(2)})v(ghg^{-1}).\]
Comparing with \cref{eq:psi-normal-inside} completes the proof.
\end{proof}

\begin{Lem}
	Any commutative direct tensor factor of $\D(G)$ is also a commutative direct tensor factor of $\tprod{G}$.
\end{Lem}
\begin{proof}
Suppose $L$ is a commutative Hopf sub-algebra of $\D(G)$ such that $\D(G)=M\otimes L$ for some Hopf sub-algebra $M$.  We then have a projection $\pi\colon\D(G)\to L$ with associated right inverse the imbedding $i\colon L\hookrightarrow\D(G)$.

The morphism $i\pi$ is an endomorphism of $\D(G)$.  Since the image is central in $\D(G)$ it is easily seen to be untwistable and binormal.  Therefore $i\pi$ is canonically a twistable, binormal, idempotent endomorphism of $\tprod{G}$ with image $L$.  By Fitting's lemma we conclude that $L$ is also a direct tensor factor of $\tprod{G}$.
\end{proof}
\begin{Rem}
Since $\du{G}$ is commutative we see that the converse will only hold when $G$ is abelian. Indeed since $\D(G)$ is quasitriangular any commutative direct tensor factor of $\D(G)$ is necessarily abelian.
\end{Rem}

The lemma gives one part of the following.
\begin{Thm}
	Let $G$ be a finite group.  Then the following are equivalent.
	\begin{enumerate}
		\item $G$ is purely non-abelian.
		\item $\kk G$ is purely non-abelian.
		\item $\du{G}$ is purely non-abelian.
		\item $\tprod{G}$ is purely non-abelian.
		\item $\D(G)$ is purely non-abelian.
	\end{enumerate}
Indeed, $\tprod{G}$ and $\D(G)$ have the same abelian direct tensor factors.
\end{Thm}
\begin{proof}
	Since the dual of an abelian Hopf algebra is again abelian, the equivalence of the second and third is immediate.  By Krull-Remak-Schmidt, any abelian indecomposable factor of $\tprod{G}$ is isomorphic to an abelian indecomposable factor of either $\kk G$ or $\du{G}$.  Thus the fourth is equivalent to the second and third.  Since any Hopf sub-algebra of $\kk G$ is a subgroup algebra, the first and second are equivalent.  By the lemma the fourth implies the fifth.  To prove the fifth implies the fourth, we need only show that any abelian factor of $\tprod{G}$ yields an abelian factor of $\D(G)$.
	
	So let $L$ be an abelian tensor factor of $\tprod{G}$ with associated projection $\pi$ and inclusion $i$.  We wish to show that $i\pi$ is canonically a binormal endomorphism of $\D(G)$.  Writing $i\pi=\begin{pmatrix} u&r\\p&v \end{pmatrix}$, the properties of $\Hopfend(\D(G))$ and commutativity of the image easily imply the following: $p\comm v$, $v\comm v$, $u\cocomm p$, $p\comm\id$, $v\comm\id$.  In particular, $v$ and $p$ have central image, and $v$ is a (bi)normal group homomorphism.  Since $(i\pi)^*$ is also an idempotent endomorphism with abelian image we similarly conclude that $p^*$ and $u^*$ have central image, and that $u^*$ is a (bi)normal group homomomorphism.  Centrality of the image of $u^*$ (indeed, that $u^*$ has abelian image and is thus a class function) implies that $G$ acts trivially on the image of $u$.  Applying the proposition we conclude that $i\pi$ is canonically a binormal endomorphism of $\D(G)$ with image $L$.  Fitting's lemma then implies that $L$ is a direct tensor factor of $\D(G)$, as desired.
	
	This completes the proof.	
\end{proof}
Thus for $G=C\times H$ with $C$ abelian and $H$ purely non-abelian we conclude that $\D(C)$ and $\D(H)$ have no common direct tensor factors.  Therefore we may apply the results of the previous section to obtain the following.
\begin{Thm}
	Let $G=C\times H$, where $C,H$ are finite groups with $C$ abelian and $H$ purely non-abelian.  Then
	\[\Hopfaut(\D(G)) = \begin{pmatrix}
      \Hopfaut(\D(C))&\Zenthom(\D(H),\D(C))\\
      \Zenthom(\D(C),\D(H))&\Hopfaut(\D(H))
    \end{pmatrix}.\]
\end{Thm}

The determination of the $\Zenthom$ terms remains a computational problem, but all of the components of these morphisms are guaranteed to be morphisms of Hopf algebras, and so determined by group homomorphisms.  Note that for $\Zenthom(\D(H),\D(C))$ we have a commutative image, as considered in \cref{prop:algebra-change-preserves}.  Whenever the field is such that $\D(C)$ is just a group algebra then the situation is further simplified.  In this case $\Zenthom(\D(C),\D(H))=\Hom(\Gamma_C,Z(\Gamma_H))$, a group of morphisms between abelian groups.

\begin{Expl}\label{ex:dihedrals}
	Consider a field $\kk$ of characteristic not $2$.  For $n\geq 3$ let $G=D_{2n}$ be the dihedral group of order $2n$, and suppose that $n\equiv 2\bmod 4$.  The group $G$ has an abelian direct factor precisely under this assumption on $n$, in which case $G\cong\Z_2\times D_n$. So we take $C=\Z_2$ and $H=D_n$, and note $\Gamma_C\cong\Z_2^2$ and $\Gamma_H\cong \Z_2\times D_n$.  It is also well-known that $\operatorname{Aut}(\Gamma_C)\cong S_3$.  By \cite{K14} we have $\Hopfaut(\D(D_{n}))\cong \Z_2\times \operatorname{Aut}(D_n)\cong \Z_2\times\operatorname{Hol}(\Z_{n/2})$.  Here $\operatorname{Hol}(\Z_n)=\Z_n\rtimes\operatorname{Aut}(\Z_n)$ is the holomorph of $\Z_n$, a group of order $n\phi(n)$, where $\phi$ is the Euler totient function.
	
	We have $Z(\Gamma_H)\cong\Z_2$, from which it follows that $\Hom(\Gamma_C,Z(\Gamma_H))\cong \Z_2^2$ as groups.  We claim that \[\Zenthom(\D(H),\D(C))\cong\Z_2^2\] as well.  Let $(u,r,p,v)\in\Hom(\D(H),\D(C))$.  The abelian normal subgroups of $D_n$ all have odd order, so $p$ is necessarily trivial.  By normality of $u^*\circ v$, we have $u^*(b^{v(x)}) = u^*(b) = u^*(b)^x$ for all $x\in D_n$ and $b\in \Z_2$.  Since no order 2 subgroup of $D_n$ is normal we conclude that $u^*$ is trivial.  From this we can then easily check that $\Hopfmor(\D(H),\D(C))=\Zenthom(\D(H),\D(C))$.  Since there are two possible homomorphisms $v\colon D_n\to\Z_2$, and two possible homomorphisms $r\colon D_n\to\widehat{\Z_2}$, all of which satisfy the necessary compatibilities, it quickly follows that $\Zenthom(\D(H),\D(C))\cong\Z_2^2$ as desired.
	
	As a consequence, $|\Hopfaut(\D(D_{2n}))| = 2^{5}\cdot 3\cdot n\cdot\phi(n/2)$ whenever $n\equiv 2\bmod 4$.  For $n=6$ the order is $1152=2^7\cdot 3^2$.  The description and order of $\Hopfaut(\D(D_{2n}))$ for $n\not\equiv 2\bmod 4$ is given in \cite{K14}.
\end{Expl}
\bibliographystyle{alpha}
\bibliography{andere,arxiv,eigene,mathscinet}

\end{document}